\documentclass[paper=a4, fontsize=11pt, abstract=on]{scrartcl}

\usepackage{amsmath}
\usepackage{amsthm}
\usepackage{amsfonts}
\usepackage{amssymb}
\usepackage{mathtools}
\usepackage{hyperref}
\usepackage{dsfont}
\usepackage{mathrsfs}
\usepackage{titling}

\usepackage[headsepline]{scrlayer-scrpage}
\usepackage{geometry}

\usepackage[utf8]{inputenc}

\usepackage{lmodern}
\usepackage[T1]{fontenc}

\title{On the Rate of Convergence to a Gamma Distribution on Wiener Space }
\author{
	E. Azmoodeh\thanks{Ruhr University Bochum, Faculty of Mathematics, IB 2/101, 44780 Bochum, Germany. E-mail: ehsan.azmoodeh@rub.de},
	P. Eichelsbacher\thanks{Ruhr University Bochum, Faculty of Mathematics, IB 2/115, 44780 Bochum, Germany. E-mail: peter.eichelsbacher@rub.de}
	and L. Knichel\thanks{Ruhr University Bochum, Faculty of Mathematics, IB 2/95, 44780 Bochum, Germany. E-mail: lukas.knichel@rub.de.
	Lukas Knichel has been supported by the German Research Foundation (DFG) via Research Training Group RTG 2131 \textit{High dimensional phenomena in probability -- fluctuations and discontinuity}}
}
\date{\today}
\theoremstyle{plain}
\newtheorem{Thm}{Theorem}[section]
\newtheorem{Lem}[Thm]{Lemma}
\newtheorem{lem}[Thm]{Lemma}
\newtheorem{Prop}[Thm]{Proposition}
\newtheorem{prop}[Thm]{Proposition}
\newtheorem{cor}[Thm]{Corollary}

\theoremstyle{definition}
\newtheorem{Rem}[Thm]{Remark}
\newtheorem{rem}[Thm]{Remark}
\newtheorem{assum}[Thm]{Assumption}
\newtheorem{ex}[Thm]{Example}

\newtheorem{con}[Thm]{Conjecture}

\newcommand{\HH}{\mathfrak{H}}

\def\E{\mathbb{E}}
\def\R{\mathbb{R}}

\def\C{\mathbb{C}}
\def\N{\mathbb{N}}

\newcommand{\cont}[1]{\mathbin{\otimes_{#1}}}
\newcommand{\contIterated}[2]{\mathbin{\otimes_{#1}^{(#2)}}}
\newcommand{\scont}[1]{\mathbin{\widetilde{\otimes}_{#1}}}
\newcommand{\tensor}{\mathbin{\otimes}}
\newcommand{\symtensor}{\mathbin{\widetilde{\otimes}}}

\DeclareMathOperator{\Var}{Var}
\DeclareMathOperator{\Tr}{Tr}

\DeclareMathOperator{\const}{const.}
\DeclareMathOperator{\CenteredGamma}{\overline{\Gamma}}
\DeclareMathOperator{\confhyper}{{}_{1}F _{1}}
\DeclarePairedDelimiter\sprod{\langle}{\rangle}
\DeclarePairedDelimiter\abs{\lvert}{\rvert}
\DeclarePairedDelimiter\norm{\lVert}{\rVert}

\newcommand{\ind}[1]{\mathds{1}_{\{ #1 \}}}

\let\temp\epsilon \let\epsilon\varepsilon \let\varepsilon\temp
\let\temp\phi \let\phi\varphi \let\varphi\temp

\begin{document}
	\maketitle
	\begin{abstract}
	In \cite{n-p-noncentral}, Nourdin and Peccati established a neat characterization of Gamma approximation on a fixed Wiener chaos in terms of convergence of only the third and fourth cumulants.  In this paper, we investigate the rate of convergence in Gamma approximation on Wiener chaos in terms of the \textit{iterated Gamma operators} of Malliavin Calculus. On the second Wiener chaos, our upper bound can be further extended to an exact rate of convergence in a suitable probability metric $d_2$
	in terms of the maximum of the third and fourth cumulants, analogous to that of normal approximation in \cite{n-p-optimal} under one extra mild condition. We end the paper with some novel Gamma characterization within the second Wiener chaos as well as Gamma approximation in Kolmogorov distance relying on the classical Berry--Esseen type inequality.  
	\end{abstract}


\section{Introduction}
Given a separable Hilbert space $\HH$, we consider an isonormal Gaussian process $X=\{X(h) : h \in \HH \}$ defined on some probability space $(\Omega, \mathscr{F} , P)$. Our object of interest is a sequence $(F_n)_{n\geq 1}$ living inside a fixed Wiener chaos of order $q$ with fixed variance, e.g. $\E[F_n^2]=1$. In recent years, these objects have been studied extensively, with one of the most famous results being the so-called fourth moment theorem, which first appeared in \cite{FmtOriginalReference}. It states that $F_n \stackrel{\mathcal{D}}{\to} N$, where $N \sim \mathscr{N}(0,1)$ is a standard normal random variable, if and only if $\E[F_n^4] \to 3$. In 2009, the authors of \cite{StMethOnWienChaos} proved a quantitative version of the fourth moment theorem combining Stein's method for normal approximation with Malliavin calculus on the Wiener space. In this paper, they provide explicit bounds for the total variation distance between $F_n$ and $N$ in terms of the fourth cumulant of $F_n$, namely 
\[ d_{TV}(F_n,N) \leq 2 \sqrt{\frac{q-1}{3q} \kappa_4(F_n)}.\]
Recall that for two random variables $X$ and $Y$, the total variation distance is $d_{TV} (X,Y) := \sup_{A \in \mathcal{B}(\R) } \abs{ P(X \in A) - P(Y \in A) }$,
where $\mathcal{B}(\R)$ is the set of all bounded Borel sets.
In \cite{n-p-optimal}, the optimal rate of convergence in the fourth moment theorem has been found. More precisely, if $F_n$ converges in law to $N$, then there exist constants $0<c<C$ (not depending on $n$), such that
\[ c \times \max\{ \abs{\kappa_3(F_n)}, \abs{\kappa_4(F_n)}  \} \leq d_{TV}(F_n,N) \leq  C \times \max\{ \abs{\kappa_3(F_n)}, \abs{\kappa_4(F_n)}  \}.\]
Note that the square root from the previous results has been removed and that the third cumulant comes into play.

Limit theorems for a Gamma target distribution have been considered e.g. in \cite{n-p-noncentral}, \cite{StMethOnWienChaos} and \cite{InvPrinForHomSums}. We consider the target $G(\nu) \sim CenteredGamma(\nu)$. This means that $G(\nu) = 2 \, \widehat{G}(\nu /2) - \nu$, where $\widehat{G}(\nu /2)$ is a standard Gamma random variable with density $\widehat{g}(x) = x^{\frac{\nu}{2} -1} \, e^{-x} \, \Gamma(\frac{\nu}{2})^{-1} \, \mathds{1}_{(0, \infty)}(x)$. From now on, we still assume $F$ to be inside a fixed Wiener Chaos of order $q \geq 2$, but fix our variance to be $\Var(F)=\Var(G(\nu))=2\nu$.

In \cite{d-p}, the authors used a Stein equation suitable for proving Stein-Malliavin upper bounds in $1$-Wasserstein distance for the convergence of $F_n$ to $G(\nu)$. In Theorem 1.7, they showed that
\[ d_1(F,G(\nu)) \leq \max \Big(1, \frac{2}{\nu} \Big) \E\Big[ \big( 2F +2\nu -  \Gamma_1(F) \big)^2 \Big]^{1/2},  \]
where $d_1$ denotes the $1$-Wasserstein distance and $\Gamma_1$ is the Gamma operator defined in the next section.

From \cite[Theorem~3.6]{InvPrinForHomSums}, for any random variable $F$ in the $q$-th Wiener chaos with $\E[F^2]=2\nu$, we  have the estimate
\begin{align}
\E\Big[ \big(2F + 2\nu  - \Gamma_1(F) \big)^2 \Big] & \leq \frac{q-1}{3q} \abs{\kappa_4(F) - \kappa_4(G(\nu)) - 12 \kappa_3(F) +12 \kappa_3(G(\nu)) } \notag  \\
& \leq \const \times \max \big\{ \abs{ \kappa_3(F) - \kappa_3(G(\nu)) } , \abs{ \kappa_4(F) - \kappa_4(G(\nu)) } \big\}.
\label{eq:Var(Gamma_1-2F)leqCumulants}
\end{align}
Combining these two results, we obtain an upper bound similar to the one in the fourth moment theorem, but worse by a whole square root, namely
\begin{equation}
d_1(F,G(\nu)) \leq \const \times \max \big\{ \abs{ \kappa_3(F) - \kappa_3(G(\nu)) } , \abs{ \kappa_4(F) - \kappa_4(G(\nu)) } \big\}^{1/2}. \label{eq:non-optimal-bound}
\end{equation}
A natural question, which we will deal with in this paper, is if this square root can be removed using techniques similar to the ones in \cite{n-p-optimal}.

As a generalization of the Wasserstein-$1$ distance $d_1$, we also define the following probability metrics. For $k \geq 1 $, let 
\[ \mathcal{H}_k := \{ h \in C^{k-1}(\R) : h^{(k-1)} \in \operatorname{Lip}(\R) \text{ and } \norm{h^{(1)}}_\infty \leq 1,  \ldots, \norm{h^{(k)}}_\infty \leq 1 \}.  \]
Furthermore define the corresponding distance between two random variables $X$ and $Y$ as
\begin{equation}
d_k(X,Y) := \sup_{h \in \mathcal{H}_k} \abs[\Big]{ \E[h(X)] - \E[h(Y)]  }. \label{eq:d_kMetricsDefinition}
\end{equation}

The outline of our paper is as follows: In section $2$, we give a brief introduction to Mallia\-vin calculus on the Wiener space and specify the notation used in the paper. The third section contains the main theoretical finding of this paper -- an upper bound for the $d_2$ distance
between a general element $F$ living in a finite sum of Wiener chaoses and our target distribution $G(\nu)$ in terms of iterated Gamma operators.
In Section \ref{sec:2wiener}, we shift our focus to the case of a random variable $F=I_2(f)$ in the second Wiener chaos to establish an optimality result similar to the main result in \cite{n-p-optimal} by removing the square root in \eqref{eq:non-optimal-bound}. Section 5 provides several new characterizations of the centered Gamma distribution $G(\nu)$ within the second Wiener chaos in terms of iterated Gamma operators.  The final section deals with a different collection of techniques, mainly based on a classical Berry-Esseen lemma, to present several Gamma approximation results in the Kolmogorov distance.

\section{Preliminaries: Gaussian Analysis and Malliavin Calculus}
In this section, we provide a very brief introduction to Malliavin calculus and define some of the operators used in this framework. For a more detailed introduction and proofs, see for example the textbooks \cite{n-pe-1} and \cite{GelbesBuch}.

\subsection{Isonormal Gaussian Processes and Wiener Chaos}
Let $\mathfrak{H}$ be a real separable Hilbert space with inner product $\sprod{\cdot,\cdot}_{\mathfrak{H}}$, and $X = \{X(h) : h \in \mathfrak{H} \}$ be an isonormal Gaussian process, defined on some probability space $(\Omega, \mathscr{F},P)$. This means that $X$ is a family of centered, jointly Gaussian random variables with covariance structure $\E[X(g)X(h)] = \sprod{g,h}_{\mathfrak{H}}$. We assume that $\mathscr{F}$ is the $\sigma$-algebra generated by $X$. For an integer $q \geq 1$, we will write $\mathfrak{H}^{\otimes q}$ or $\mathfrak{H}^{\odot q}$ to denote the $q$-th tensor product of $\mathfrak{H}$, or its symmetric $q$-th tensor product, respectively. If $H_q(x) =(-1)^{q}e^{x^{2}/2}{\frac {d^q}{d x^n}}e^{-x^{2}/2}$ is the $q$-th Hermite polynomial, then the closed linear subspace of $L^2(\Omega)$ generated by the family $\{H_q(X(h)) : h \in \mathfrak{H}, \norm{h}_\mathfrak{H} = 1 \}$ is called the $q$-th \textit{Wiener chaos} of $X$ and will be denoted by $\mathscr{H}_q$. For $f \in \mathfrak{H}^{\odot q}$, let $I_q(f)$ be the $q$-th multiple Wiener-Itô integral of $f$ (see \cite[Definition~2.7.1]{n-pe-1}). An important observation is that for any $f \in \mathfrak{H}$ with $\norm{f}_{\mathfrak{H}}=1$ we have that $H_q(X(f)) = I_q(f^{\otimes q})$. As a consequence $I_q$ provides an isometry from $\mathfrak{H}^{\odot q}$ onto the $q$-th Wiener chaos $\mathscr{H}_q$ of $X$. It is a well-known fact, called the \textit{Wiener-Itô chaotic decomposition}, that any element $F \in L^2(\Omega)$ admits the expansion
\begin{equation}
F = \sum_{q=0}^{\infty} I_q(f_q), \label{eq:ChaoticExpansion}
\end{equation}
where $f_0 = \E[F]$ and the $f_q \in \mathfrak{H}^{\odot q}$, $q \geq 1$ are uniquely determined. An important result is the following isometry property of multiple integrals. Let $f \in \mathfrak{H}^{\odot p}$ and $g \in \mathfrak{H}^{\odot q}$, where $1 \leq q \leq p$. Then
\begin{equation}
\E[ I_p(f) I_q(g) ] = \begin{cases}
p! \, \sprod{f,g}_{\mathfrak{H}^{\otimes p}}  & \text{if } p= q\\
0 & \text{otherwise}.
\end{cases} \label{eq:IsometryProperty}
\end{equation}

\subsection{The Malliavin Operators}

We denote by $\mathscr{S}$ the set of \textit{smooth} random variables, i.e. all random variables of the form $F= g(X(\varphi_1),\ldots, X(\varphi_n))$, where $n \geq 1$, $\varphi_1, \ldots, \varphi_n \in \HH$ and $g:\R^n \to \R$ is a $C^{\infty}$-function, whose partial derivatives have at most polynomial growth. For these random variables, we define the \textit{Malliavin derivative} of $F$ with respect to $X$ as the $\HH$-valued random element $DF \in L^2(\Omega,\HH)$ defined as
\[ DF = \sum_{i=1}^{\infty} \frac{\partial g}{\partial x_i} \big( X(\varphi_1), \ldots, X(\varphi_n) \big) \, \varphi_i. \]
The set $\mathscr{S}$ is dense in $L^2(\Omega)$ and using a closure argument, we can extend the domain of $D$ to $\mathbb{D}^{1,2}$, which is the closure of $\mathscr{S}$ in $L^2(\Omega)$ with respect to the norm $\norm{F}_{\mathbb{D}^{1,2}} := \E[F^2] + \E[ \norm{DF}_{\HH}^2 ]$. See \cite{n-pe-1} for a more general definition of higher order Malliavin derivatives and spaces $\mathbb{D}^{p,q}$. The Malliavin derivative satisfies the following chain-rule. If $\phi:\R^m \to \R$ is a continuously differentiable function with bounded partial derivatives and $F=(F_1, \ldots, F_m)$ is a vector of elements of $\mathbb{D}^{1,q}$ for some $q$, then $\phi(F)\in \mathbb{D}^{1,q}$ and
\begin{equation}
D \phi(F) = \sum_{i=1}^{m} \frac{\partial \phi}{\partial x_i} (F) \, D F_i. \label{eq:ChainRule}
\end{equation}
Note that the conditions on $\phi$ are not optimal and can be weakened. For $F \in L^2(\Omega)$, with chaotic expansion as in \eqref{eq:ChaoticExpansion}, we define the \textit{pseudo-inverse} of the infinitesimal generator of the Ornstein-Uhlenbeck semigroup as
\[ L^{-1} F = - \sum_{p=1}^{\infty} \frac{1}{p} I_p(f_p). \]

The following integration by parts formula is one of the main ingredients to proving the main theorem of section \ref{sec:MainUpperBound}. Let $F,G \in \mathbb{D}^{1,2}$. Then 
\begin{equation}
\E[FG] = \E[F] \E[G] + \E[ \sprod{DG, -DL^{-1}F}_{\HH} ].  \label{eq:IntegrationByParts}
\end{equation}

\subsection{Gamma Operators and Cumulants}

Let $F$ be a random variable with characteristic function $\phi_F(t) = \E[ e^{itF}]$. We define its $n$-th cumulant, denoted by  $\kappa_n(F)$, as
\[ \kappa_n(F) = \frac{1}{i^n} \frac{\partial^n}{\partial t^n} \log \phi_F(t) \Big\vert_{t=0}. \]
Let $F$ be a random variable with a finite chaos expansion. We define the operators $\Gamma_i$, $i \in \N_0$ via
\[ \Gamma_0(F) := F \]
and
\begin{equation}
\Gamma_{i+1} (F) := \sprod{D \Gamma_{i}(F) , -D L^{-1} F}_{\mathfrak{H}}, \quad \text{for } i\geq 0. \label{eq:GammOperatorDefinition}
\end{equation}
This is the Gamma operator used in the proof of the main theorem in \cite{n-p-optimal}, although it is defined differently there.
Note that there is also an alternative definition, which can be found in most other papers in this framework, see for example Definition 8.4.1 in \cite{n-pe-1} or Definition 3.6 in \cite{OptBerryEsseenRates}. For the sake of completeness, we also mention the classical Gamma operators, which we also call \textit{alternative} Gamma operators, which we shall denote by $\Gamma_{alt}$. These are defined via
\begin{equation}
\Gamma_{alt, 0}(F) := F \quad \text{and} \quad \Gamma_{alt, i+1} (F) := \sprod{D F , -D L^{-1} \Gamma_{alt, i}(F)}_{\mathfrak{H}}, \quad \text{for } i\geq 0. \label{eq:GammaOperatorAltDefinition}
\end{equation}
The classical Gamma operators are related to the cumulants of $F$ by the following identity from \cite{CumOnTheWienerSpace}: For all $j \geq 0$, we have
\[ \E[\Gamma_{alt, j}(F)] = \frac{1}{j!} \kappa_{j+1}(F). \]
If $j \geq 3$, this does not hold anymore for our new Gamma operators. Instead, we will list some useful relations between the classical and the new Gamma operators.

\begin{Prop} \label{Prop:RelationOldAndNewGamma}
	Let $F$ be a centered random variable admitting a finite chaos expansion. Then
\begin{itemize}
	\item[(a)] $\Gamma_1(F) = \Gamma_{alt,1}(F)$,
	\item[(b)] if $j=1$ or $j=2$, then
	\[ \E\big[ \Gamma_j(F) \big] = \E\big[ \Gamma_{alt,j}(F) \big] = \frac{1}{j!} \kappa_{j+1}(F), \]
	\item[(c)] $ \E\big[ \Gamma_3(F) \big] = 2 \, \E \big[ \Gamma_{alt,3}(F) \big] - \Var\big(\Gamma_1(F) \big) = \frac{1}{3} \kappa_4(F) - \Var\big(\Gamma_1(F) \big)$,
	\item[(d)] When $F= I_2(f)$, for some $f \in \HH^{\odot 2}$, is an element of the second Wiener chaos, then
	\[ \Gamma_j(F) = \Gamma_{alt,j}(F) \quad \text{for all } j \geq 1. \]
\end{itemize}
\end{Prop}
The proofs of these statements can be found in the appendix along with an explicit representation of the Gamma operators in terms of contractions.

\section{The Stein-Malliavin Upper Bound} \label{sec:MainUpperBound}
In the following, we will use centered versions of the Gamma-operators, i.e.
\[  \CenteredGamma_j(F)  := \Gamma_j(F) - \E[\Gamma_j(F)] =  \Gamma_j(F) - \frac{1}{j!} \kappa_{j+1}(F) .\]

\begin{Thm}\label{thm:MainMalliavinSteinBound}
	Let $F$ be a centered random variable admitting a finite chaos expansion with $\Var(F) = 2 \nu$. Let $G(\nu) \sim CenteredGamma(\nu)$. Then there exists a constant $C>0$ (only depending on $\nu$), such that
	\begin{align}\label{main-theoretical-estimate}
	d_2(F,G(\nu)) \leq C \, \bigg\{ & \E \left[ \big( 2F - \CenteredGamma_1(F) \big)^2 \right] \notag  + \E \left[ \big( \CenteredGamma_2(F) - 2 \CenteredGamma_1(F) \big)^2 \right]^{1/2} \E \left[ \big( 2F - \CenteredGamma_1(F) \big)^2 \right]^{1/2} \notag \\
	& + \E \left[ \big( \CenteredGamma_3(F) - 2 \CenteredGamma_2(F) \big)^2 \right]^{1/2} \notag \\
	& + \abs*{\kappa_3(F) - \kappa_3(G(\nu)) } +  \abs*{\kappa_4(F) - \kappa_4(G(\nu)) } \bigg\}.
	\end{align}
\end{Thm}

To simplify computations, we begin with the following Lemma.

\begin{Lem} \label{lem:AuxiliarlyLemmaForMainTheorem}
	Let $g : \R \to \R$ be a Lipschitz function, where $g$ and $g'$ are bounded by a constant only depending on $\nu > 0$. Consider the solution $\phi$ to the Stein equation $g(x) - \E[g(G(\nu))] = 2(x+\nu) \phi'(x) - x \phi(x) $. Then
	\begin{itemize}
		\item[(a)] $\phi$ is again a Lipschitz function, where $\phi$ and $\phi'$ are bounded by a constant only depending on $\nu$.
		\item[(b)] If $F\in \mathbb{D}^{\infty}$ is a centered random variable with variance $\E[F^2]= 2 \nu$, then for any $r \in \N$:
		\begin{align*}
		\E \Big[ g(F) \Big( \CenteredGamma_r(F) - 2 \CenteredGamma_{r-1}(F) \Big) \Big] = & - \E\Big[ \phi'(F) \Big( \CenteredGamma_r(F) - 2 \CenteredGamma_{r-1}(F) \Big) \Big( \CenteredGamma_1(F) - 2 F \Big) \Big] \\
		& - \E \Big[ \phi(F) \Big( \Gamma_{r+1}(F) - 2 \Gamma_r(F) \Big) \Big].
		\end{align*}
	\end{itemize}
\end{Lem}

\begin{proof}
	Part \textit{(a)} is a consequence of \cite[Theorem~2.3~(a)]{d-p}. For part \textit{(b)}, note that $2\nu = \E[ \Gamma_1(F)]$. Thus
	\begin{align*}
	\E \Big[ g(F) \Big( \CenteredGamma_r(F) - 2 \CenteredGamma_{r-1}(F) \Big) \Big] & = \E \Big[ \Big( g(F) - \E\big[g(G(\nu)) \big] \Big)  \Big( \CenteredGamma_r(F) - 2 \CenteredGamma_{r-1}(F) \Big) \Big] \\
	& = \E \Big[ \Big( 2(F+\nu) \phi'(F) - F \phi(F) \Big) \Big( \CenteredGamma_r(F) - 2 \CenteredGamma_{r-1}(F) \Big) \Big]  \\
	& \hspace{-10em} = 
	2 \, \E\big[ F \phi'(F) \CenteredGamma_r(F) \big]
	+ \E \big[ \Gamma_1(F) \big] \E \big[\phi'(F) \CenteredGamma_r(F) \big]
	- \E \big[ F \phi(F) \CenteredGamma_r(F) \big] \\
	& \hspace{-9em} 
	- 4 \, \E\big[ F \phi'(F) \CenteredGamma_{r-1}(F) \big]
	- 2 \, \E\big[ \Gamma_1(F) \big] \E \big[ \phi'(F) \CenteredGamma_{r-1}(F) \big]
	+ 2 \, \E \big[ F \phi(F) \CenteredGamma_{r-1}(F) \big] \\
	& = : \sum_{i=1}^{6} T_i.
	\end{align*}
	Now, we use the integration-by-parts formula \eqref{eq:IntegrationByParts}  in combination with the chain rule \eqref{eq:ChainRule} to obtain
	\begin{align*}
	T_3 + T_2 & = - \E \big[ F \phi(F) \CenteredGamma_r(F) \big] + \E \big[ \Gamma_1(F) \big] \E \big[\phi'(F) \CenteredGamma_r(F) \big] \\
	& = - \E \big[ \Gamma_1(F) \CenteredGamma_r(F) \phi'(F) \big] - \E \big[ \phi(F) \Gamma_{r+1}(F) \big] + \E \big[ \Gamma_1(F) \big] \E \big[\phi'(F) \CenteredGamma_r(F) \big] \\
	& = - \E \big[ \CenteredGamma_1(F) \CenteredGamma_r(F) \phi'(F) \big] - \E \big[ \phi(F) \Gamma_{r+1}(F) \big],
	\end{align*}
	and similarly
	\[
	T_6 + T_5 = 2 \, \E \big[ \CenteredGamma_1(F) \CenteredGamma_{r-1}(F) \phi'(F) \big] + 2 \, \E \big[ \phi(F) \Gamma_r(F) \big].
	\]
	Hence, putting everything together, the result follows.
\end{proof}

\begin{proof}[Proof of Theorem \ref{thm:MainMalliavinSteinBound}]

	As a starting point, we use the Stein equation (2.7) from \cite{d-p}. 
	Let $h \in \mathcal{H}_2$ be a test function, then by using the integration by parts formula \eqref{eq:IntegrationByParts}, we get
	\begin{align*}
	\abs*{\E[h(F)] - \E[h(G(\nu))] } & = \abs*{ \E \left[ 2(F+ \nu) \, f'(F) - F f(F) \right] } \\
	& = \abs*{ \E \left[ 2(F+ \nu) \, f'(F) - f'(F) \sprod{DF, - D L^{-1} F }_{\mathfrak{H}} \right] } \\
	& = \abs*{ \E \left[ f'(F) \left( 2F - \CenteredGamma_1(F) \right)\right] }.
	\end{align*}
Now set $g:=f'$. Then $g$ is a bounded Lipschitz function whose derivative $g'$ is bounded by a constant only depending on $\nu$, see \cite[Theorem~2.3~(b)]{d-p}.
Denote by $\phi$ the solution to the Gamma Stein equation $g(x) - \E[g(G(\nu))] = 2(x+\nu) \phi'(x) - x \phi(x) $, and by $\psi$ the solution to $\phi(x) - \E[\phi(G(\nu))] = 2(x+\nu) \psi'(x) - x \psi(x) $. By Lemma \ref{lem:AuxiliarlyLemmaForMainTheorem} \textit{(a)}, both $\phi$ and $\psi$ are Lipschitz, where the functions themselves, as well as their derivatives are bounded by a constant only depending on $\nu$. Now apply Lemma \ref{lem:AuxiliarlyLemmaForMainTheorem} \textit{(b)} twice, to get
\begin{align*}
\E & \Big[ g(F) \Big( 2 F - \CenteredGamma_{1}(F) \Big) \Big] = \E\Big[ \phi'(F) \Big( \CenteredGamma_1(F) - 2 F \Big)^2 \Big] + \E \Big[ \phi(F) \Big( \Gamma_{2}(F) - 2 \Gamma_1(F) \Big) \Big] \\
& {}={} \E\Big[ \phi'(F) \Big( \CenteredGamma_1(F) - 2 F \Big)^2 \Big] - \E [\phi(F)] \Big( \frac{1}{2} \kappa_3(F) - 2 \kappa_2(F) \Big) \\
& \quad + \E \Big[ \phi(F) \Big( \CenteredGamma_{2}(F) - 2 \CenteredGamma_1(F) \Big) \Big] \\
& {}={} \E\Big[ \phi'(F) \Big( \CenteredGamma_1(F) - 2 F \Big)^2 \Big] - \E [\phi(F)] \Big( \frac{1}{2} \kappa_3(F) - 2 \kappa_2(F) \Big) \\
& \quad - \E\Big[ \psi'(F) \Big( \CenteredGamma_2(F) - 2 \CenteredGamma_{1}(F) \Big) \Big( \CenteredGamma_1(F) - 2 F \Big) \Big] - \E \Big[ \psi(F) \Big( \Gamma_{3}(F) - 2 \Gamma_2(F) \Big) \Big] \\
& {}={} \E\Big[ \phi'(F) \Big( \CenteredGamma_1(F) - 2 F \Big)^2 \Big] - \E\Big[ \psi'(F) \Big( \CenteredGamma_2(F) - 2 \CenteredGamma_{1}(F) \Big) \Big( \CenteredGamma_1(F) - 2 F \Big) \Big] \\
& \quad - \E \Big[ \psi(F) \Big( \CenteredGamma_{3}(F) - 2 \CenteredGamma_2(F) \Big) \Big] - \E [\phi(F)] \Big( \frac{1}{2} \kappa_3(F) - 2 \kappa_2(F) \Big) \\
& \quad + \E \big[ \psi(F) \big] \Big( \E \big[\Gamma_3(F) \big] - \kappa_3(F) \Big).
\end{align*}
Note that we cannot translate $\E[\Gamma_3(F)]$ directly into the fourth cumulant, but instead by Proposition \ref{Prop:RelationOldAndNewGamma} part \textit{(c)} we have $\E[\Gamma_3(F)] = \frac{1}{3} \kappa_4(F) - \Var(\Gamma_1(F))$. The variance term can be written as
\begin{align*}
\Var\big( \Gamma_1(F) \big) & = \Var\big( \Gamma_1(F) - 2 F \big) - 4 \kappa_2(F) + 4 \E\big[ F \Gamma_1(F) \big] \\
& = \Var\big( \Gamma_1(F) - 2 F \big) - 4 \kappa_2(F) + 4 \E\big[ \Gamma_2(F) \big] \\
& = \Var\big( \Gamma_1(F) - 2 F \big) - 4 \kappa_2(F) + 2 \kappa_3(F).
\end{align*}
Putting everything together, we obtain
\begin{align*}
\E & \Big[ g(F) \Big( 2 F - \CenteredGamma_{1}(F) \Big) \Big] \\
& {}={} \E\Big[ \phi'(F) \Big( \CenteredGamma_1(F) - 2 F \Big)^2 \Big] - \E\Big[ \psi'(F) \Big( \CenteredGamma_2(F) - 2 \CenteredGamma_{1}(F) \Big) \Big( \CenteredGamma_1(F) - 2 F \Big) \Big] \\
& \quad - \E \Big[ \psi(F) \Big( \CenteredGamma_{3}(F) - 2 \CenteredGamma_2(F) \Big) \Big] - \E [\phi(F)] \Big( \frac{1}{2} \kappa_3(F) - 2 \kappa_2(F) \Big) \\
& \quad - \E \big[ \psi(F) \big] \Var\big( \Gamma_1(F) - 2 F \big) + \E \big[ \psi(F) \big] \Big(  \frac{1}{3} \kappa_4(F) - 3 \kappa_3(F) + 4 \kappa_2(F) \Big).
\end{align*}
\\
The result follows by applying Cauchy-Schwarz inequality, as well as using the fact that $\kappa_2(G(\nu)) = \kappa_2(F) = 2 \nu$, $\kappa_3(G(\nu))= 8\nu$ and $\kappa_4(G(\nu)) = 48 \nu$, see \eqref{eq:CenteredGammaCumulants}.
\end{proof}

\begin{Rem}
\begin{itemize}
	\item[\textit{(i)}] The argument based on iterating the Stein equation, instead of applying Cauchy-Schwarz inequality after using the Malliavin integration by parts formula once, implemented in the proof of Theorem \ref{thm:MainMalliavinSteinBound}, is completely analogous to the main result from \cite[p.~3129]{n-p-optimal}. The backbone of this line of arguments is the fact that when applying higher order Gamma operators on chaotic random variables, the resulting random variables often become smaller (in terms of variance).
		\item[\textit{(ii)}] A natural framework in which to apply our main Theorem \ref{thm:MainMalliavinSteinBound}, is when the candidate random variable $F$ is chaotic, meaning that $F=I_q(f)$ for some $q \ge 2$, and kernel $f \in \mathfrak{H}^{\odot q}$. In this framework, it is well-known (e.g. \cite{InvPrinForHomSums}) that the first summand in the RHS of estimate \eqref{main-theoretical-estimate} can be further controlled by using the third and fourth cumulants, namely that
	\[  \E\left[ \left( 2F - \CenteredGamma_1(F) \right)^2 \right] = \Var\left( \Gamma_1(F) - 2 F \right)  \leq \frac{q-1}{3q} \,  \{ \kappa_4(F) -12 \kappa_3(F) + 48 \nu \}. \]
	We emphasize that, when $q \ge 4$ and $F$ is chaotic, the linear combination of the cumulants $\kappa_4(F) -12 \kappa_3(F) + 48 \nu $ is positive, see \cite[Corollary 4.4]{n-p-noncentral}.
	\item[\textit{(iii)}] In order to interpret our upper bound in Theorem \ref{thm:MainMalliavinSteinBound} in the language of the cumulants, analogous to that achieved in \cite{n-p-optimal}, for a chaotic random variable  $F=I_q(f)$ with $q \ge 2$, we need cumulant-type inequalities comparable to Proposition 4.3 in \cite{OptBerryEsseenRates} for the remaining terms in the RHS of \eqref{main-theoretical-estimate};
	\[  \Var  \big( \Gamma_2(F) - 2 \, \Gamma_1(F)  \big) \quad \text{and} \quad \Var  \big( \Gamma_3(F) - 2 \, \Gamma_2(F)  \big)  . \]
	This is a deep problem and for the time being, such inequalities are difficult to tackle in full generality using the available techniques such as \textit{contraction operators} or the recent machinery of \textit{Markov triplets} \cite{ledoux4MT}\cite{a-c-p}\cite{a-m-m-p}.  For example, a suitable cumulant counterpart for studying the variance of the iterated Gamma quantity $\Gamma_3(F)$ is the $8$th cumulant $\kappa_8(F)$. There exists an explicit representation (see \cite[Proposition 3.9]{OptBerryEsseenRates}) of the quantity $\Gamma_3(F)$ in terms of appropriate contractions, involving the kernel function $f$. However, due to several zero-contractions appearing in the cumulant side $\kappa_8(F)$ which do not show up in the Wiener chaotic representation of $\Gamma_3(F)$, such comparison is impossible.  The second major obstacle is that one needs to control the variance of an explicit linear structure of Gamma operators in terms of an ``efficient'' linear combination of cumulants. Here with ``efficient'' we mean that when plugging in the target random variable $G(\nu)$, the introduced cumulant combination vanishes. For instance, one has to note that  $\kappa_4(G(\nu)) -12 \kappa_3(G(\nu)) + 48 \nu = 0$. Thus, in Section \ref{sec:2wiener}, in order to analyze these variance quantities, we will focus on the case of $F$ belonging to the second Wiener chaos, where we have explicit representations in terms of the eigenvalues of the corresponding Hilbert-Schmidt operator.
\end{itemize}
\end{Rem}

As we will see later, when focussing on second Wiener chaos, the most critical term to analyze is $\Var  \left( \Gamma_3(F) - 2 \, \Gamma_2(F)  \right)$. If we choose our test function $h$ to be smoother, we can deduce an upper bound in the smoother probability metric $d_3$ (see  \eqref{eq:d_kMetricsDefinition}  for definition) without the need of iterating Stein's method. 

\begin{Prop}\label{Prop:H3-metric}
	Let $F$ be a centered random variable admitting a finite chaos expansion with $\Var(F) = 2 \nu$.  Let $G(\nu) \sim CenteredGamma(\nu)$. Then there exists a constant $C>0$ (only depending on $\nu$), such that
	\begin{equation}\label{eq:SecondVersionOfTheorem}	
	d_{3} (F,G(\nu)) \leq C \bigg\{ \sqrt{\Var( \CenteredGamma_{alt,3}(F) - 2 \CenteredGamma_{alt,2}(F) )} + \abs{\kappa_3(F) - \kappa_3(G(\nu))} + \abs{\kappa_4(F) - \kappa_4(G(\nu))} \bigg\} .
	\end{equation}
\end{Prop}
\begin{proof}
	Let $h \in \mathscr{H}_3$ be a test function and denote by $f$ the solution to the Stein equation $h(x) - \E[h(G(\nu))] = 2(x+\nu) f'(x) - x f(x) $. Now we use the Malliavin integration by parts formula \eqref{eq:IntegrationByParts} a total number of three times to get
	\begin{align*}
	& \E[h(F)] - \E[h(G(\nu))]  = \E \left[ 2(F+\nu) f'(F) - F f(F) \right] \\
	& \quad = \E \left[ f'(F) \left( 2F - \CenteredGamma_{alt,1}(F) \right) \right] \\
	& \quad = \E \left[ f''(F) \left( 2 \Gamma_{alt,1}(F) - \Gamma_{alt,2}(F) \right) \right] \\
	& \quad = \E \left[ f''(F) \left( 2 \CenteredGamma_{alt,1}(F) - \CenteredGamma_{alt,2}(F) \right) \right] + \E [f''(F)] \Big( 2 \E[\Gamma_{alt,1}(F)] - \E[\Gamma_{alt,2}(F)] \Big) \\
	& \quad = \E \left[ f'''(F) \left( 2 \Gamma_{alt,2}(F) - \Gamma_{alt,3}(F) \right) \right] + \E[f''(F)] \Big( \frac{1}{2} \kappa_3(F) - 4 \nu \Big) \\
	& \quad = \E \left[ f'''(F) \left( 2 \CenteredGamma_{alt,2}(F) - \CenteredGamma_{alt,3}(F) \right) \right] + \E[f'''(F)] \Big( 2 \E[\Gamma_{alt,2}(F)]  - \E[\Gamma_{alt,3}(F)] \Big) \\
	& \hspace{17.5em} + \E[f''(F)] \Big( \frac{1}{2} \kappa_3(F) - 4 \nu \Big) \\
	& \quad = \E \left[ f'''(F) \left( 2 \CenteredGamma_{alt,2}(F) - \CenteredGamma_{alt,3}(F) \right) \right] + \E[f'''(F)] \big( \kappa_3(F)  - 8 \nu \big) \\
	& \hspace{2.5em}  +  E[f'''(F)] \Big( 8 \nu - \frac{1}{6} \kappa_4(F) \Big) + \E[f''(F)] \Big( \frac{1}{2} \kappa_3(F) - 4 \nu \Big).
	\end{align*}
	We know that $\E[\kappa_3(G(\nu))] = 8 \nu$ and $\E[\kappa_4(G(\nu))] = 48 \nu$. Combining this with the boundedness of $f''$, $f'''$ and Cauchy-Schwarz inequality, we obtain
	\begin{align*} 
	&  \abs*{\E[h(F)] - \E[h(G(\nu))]} \\
	\leq {} & {} C \, \bigg\{  \E \left[ \abs*{ \left( \CenteredGamma_{alt,3}(F) - 2 \CenteredGamma_{alt,2}(F) \right) } \right] + \abs{\kappa_3(F) - \kappa_3(G(\nu))} + \abs{\kappa_4(F) - \kappa_4(G(\nu))} \bigg\} \\
	\leq {} & {} C \, \Big\{  \sqrt{\Var( \Gamma_{alt,3}(F) - 2 \Gamma_{alt,2}(F) )} + \abs{\kappa_3(F) - \kappa_3(G(\nu))} + \abs{\kappa_4(F) - \kappa_4(G(\nu))} \Big\}. 
	\end{align*}
\end{proof}
\begin{Rem}
Here, we have used the traditional Gamma operators as defined in \eqref{eq:GammaOperatorAltDefinition}. This way, we get a simple proof for the upper bound in a smoother integral probability metric. In the next section, we will focus only on random elements $F$ belonging to the second Wiener chaos, and it can be checked that in this setup the two definitions of Gamma operators coincide.
\end{Rem}

\section{The Case of Second Wiener Chaos}\label{sec:2wiener}

Throughout this section we assume that $F=I_2(f)$, for some $f \in \mathfrak{H}^{\odot 2}$, belongs to the second Wiener chaos. It is a classical result (see \cite[Section 2.7.4]{n-pe-1}) that these kind of random variables can be analyzed through the associated \textit{Hilbert-Schmidt operator }
\[ A_f : \mathfrak{H} \to \mathfrak{H}, \quad g \mapsto f \cont{1} g. \]
Denote by $\{ c_{f,i} : i \in \N \}$ the set of eigenvalues of $A_f$. We also introduce the following sequence of auxiliary kernels
\[ \Big\{ f \contIterated{1}{p} f : p \geq 1 \Big\} \subset \HH^{\odot 2}, \]
defined recursively as $ f \contIterated{1}{1} f = f$, and, for $p \geq 2$,
\[ f \contIterated{1}{p} f = \Big( f \contIterated{1}{p-1} f \Big) \cont{1} f.  \]

\begin{Prop} (see e.g. \cite[p.~43]{n-pe-1})\mbox{} \\[-1em]\label{Prop:2choas-properties}
	 \begin{enumerate}
	 	\item The random element $F$ admits the representation \begin{equation} F = \sum_{i=1}^{\infty} c_{f,i} \left( N_i^2 - 1 \right), \label{eq:SecondWienerChaosEigenvalueRepresentation} \end{equation}
	 	where the $N_i$ are i.i.d. $\mathscr{N}(0,1)$ and the series converges in $L^2(\Omega)$ and almost surely.
	 	\item For every $p \geq 2$
	 	\begin{equation}
		\begin{split} 
		\kappa_p(F)  &=2^{p-1} (p-1)! \sum_{i=1}^{\infty} c_{f,i}^p  = 2^{p-1} (p-1)!  \langle f, f \contIterated{1}{p-1} f \rangle_{\HH}\\
		&= 2^{p-1} (p-1)! \Tr \left( A^p_f \right) \label{eq:SecondWienerChaosCumulantFormulaEigenvalues} \end{split}\end{equation}
		where $ \Tr ( A^p_f )$ stands for the trace of the $p$th power of operator $A_f$.
	 \end{enumerate}
\end{Prop}

When $\nu$ is an integer i.e. $G(\nu)$ is a centered $\chi^2$ random variable with $\nu$ degrees of freedom, then \eqref{eq:SecondWienerChaosEigenvalueRepresentation} shows us that $G(\nu)$ is itself an element of the second Wiener chaos, where $\nu$-many of the eigenvalues are $1$ and the remaining ones are $0$. Hence, in this case, we deduce from \eqref{eq:SecondWienerChaosCumulantFormulaEigenvalues} that $\kappa_p(G(\nu)) = 2^{p-1} (p-1)! \, \nu$. Perhaps not surprisingly, this is also the case, when $\nu$ is any positive real number.

\begin{Lem}
	Let $\nu>0$ and $G(\nu) \sim CenteredGamma(\nu)$. Then 
	\begin{equation} 
	\kappa_p(G(\nu)) = \begin{dcases}
	0 & , p=1; \\
	2^{p-1} (p-1)! \, \nu &, p \geq 2.
	\end{dcases}  \label{eq:CenteredGammaCumulants}
	\end{equation}
\end{Lem}
\begin{proof}
	Since the cumulant generating function of a Gamma random variable is well-known, we can easily compute that of $G(\nu)$ to be
	\[ K(t) = \frac{\nu}{2} \log \left( \frac{1}{1- 2t} \right) - \nu t. \]
	By simple induction over $p$, we obtain
	\[ \frac{\mathrm{d}^p K}{\mathrm{d} t^p} (t) = \begin{dcases}
	- \nu + \frac{\nu}{1-2t} & , p=1; \\
	\frac{\nu}{2} \frac{2^p (p-1)!}{(1-2t)^{p+1}}& , p \geq 2.
	\end{dcases}   \]
	The result now follows by letting $t=0$.
\end{proof}

\begin{Lem} \label{lem:Var(Gamma_r-2Gamma_r-1)}
	Let $F=I_2(f)$, for some $f \in \HH^{\odot 2}$, and denote by $A_f$ the corresponding Hilbert-Schmidt operator with eigenvalues $\{ c_{f,i} : i \geq 1 \}$. Then for $r \geq 1$
	\begin{align*}
	\E \Big[ \Big( \CenteredGamma_r(F) - 2 \CenteredGamma_{r-1}(F) \Big)^2 \Big] & = 2^{2r+1} \sum_{i=1}^{\infty} c_{f,i}^{2r} ( c_{f,i} - 1 )^2 \\
	& = \frac{1}{(2r+1)!}  \kappa_{2r+2}(F) - \frac{4}{(2r)!}  \kappa_{2r+1}(F) + \frac{4}{(2r-1)!}  \kappa_{2r}(F).
	\end{align*}
\end{Lem}

\begin{proof}
	From \cite{a-p-p} equation (24), which follows by induction on $r$, we have the representation 
	\begin{equation}\label{eq:CentredGammaInTermsOfContraction}
	 \CenteredGamma_r(F) = 2^r I_2 \Big(f \contIterated{1}{r+1} f \Big).
	 \end{equation}
	Using the isometry property \eqref{eq:IsometryProperty}, we obtain
	\begin{align*}
	\Var \Big( \Gamma_r(F) - & 2 \Gamma_{r-1}(F) \Big) = 2^{2r+1} \, \norm{ f \contIterated{1}{r+1} f - f \contIterated{1}{r} f }_{\HH^{\otimes 2}}^{2}  \\
	& = 2^{2r+1} \Big( \sprod{f, f \contIterated{1}{2r+1} f}_{\HH^{\otimes 2}} - 2 \, \sprod{f, f \contIterated{1}{2r} f}_{\HH^{\otimes 2}} + \sprod{f, f \contIterated{1}{2r-1} f}_{\HH^{\otimes 2}} \Big)  \\
	& = 2^{2r+1} \Tr \Big( A_f^{2r+2} - 2 \, A_f^{2r+1} + A_f^{2r} \Big).
	\end{align*}
	The result now follows with \eqref{eq:SecondWienerChaosCumulantFormulaEigenvalues}.
\end{proof}

\begin{assum} \label{assum:OrderedEigenvalues}
	Since the order of the eigenvalues does not influence the distribution of the corresponding random variable, it will be handy to order them by descending absolute value. If any eigenvalue occurs with both positive and negative sign, we take the positive value first to make the representation unique. Hence for any element of the second Wiener chaos $F=I_2(f)$,
	we have a canonical representation
	\[ F {} \stackrel{\mathcal{D}}{=} {} \sum_{i=1}^{\infty} c_{f,i} \left( N_i^2 - 1 \right), \]
	where $\abs{c_{f,1}} \geq \abs{c_{f,2}} \ldots$ and if $ \abs{c_{f,i}} = \abs{c_{f,i+1}} $ for some $i \in \N$, then $ c_{f,i} \geq c_{f,i+1}$.
\end{assum}

\subsection{Motivating Examples}\label{sec:examples}

Let $\nu >0$. Assume that $\{F_n\}_{n \ge 1}$ is a sequence of random elements in the second Wiener chaos such that $\E(F^2_n)=2 \nu$ for all $n \ge 1$. The main Theorem \ref{thm:MainMalliavinSteinBound} reads that there exists a general constant $C$, such that

\begin{align*}
	d_{2}(F_n,G(\nu))  \leq_C \, \Bigg\{  & \Var\left( \Gamma_1 (F_n) - 2 F_n\right) \\
	& +  \sqrt{ \Var \left( \Gamma_2(F_n) - 2 \Gamma_1 (F_n) \right)} \times \sqrt{  \Var\left( \Gamma_1 (F_n) - 2 F_n\right)} \\
	& \hspace{-2em} + \sqrt{\Var\left( \Gamma_3 (F_n) - 2 \Gamma_2(F_n) \right)  }+ \abs{\kappa_3(F_n)  - \kappa_3(G(\nu))} + \abs{\kappa_4(F_n) - \kappa_4(G(\nu)) } \Bigg\}.
\end{align*}
	
As a consequence, in order for the square root in the upper bound in \eqref{eq:non-optimal-bound} to be removed, it is sufficient to verify the following statement. \textit{There exists a constant $C$ (independent of $n$, but may possibly depending on the sequence $\{F_n\}_{n \ge 1}$), such that the following variance-estimates hold:}

\begin{align}
\Var \left( \Gamma_2(F_n) - 2 \Gamma_1 (F_n) \right) &\le_C  \Var\left( \Gamma_1 (F_n) - 2 F_n\right), \label{variance-estimate-1}\\
\Var\left( \Gamma_3 (F_n) - 2 \Gamma_2(F_n) \right) & \le_C  \Var \, ^2 \left( \Gamma_1 (F_n) - 2 F_n \right) \label{variance-estimate-2}
\end{align}
Our major aim in the present section is to show that
\begin{itemize}
\item[\textit{(i)}] The variance-estimate \eqref{variance-estimate-1} is \textit{universal} in the sense that it holds for any random variable $F$ in the second Wiener chaos having second moment $\E(F^2)=2 \nu$. In particular it is not a matter of the fact whether the sequence $F_n$ converges in distribution towards a centered Gamma target $G(\nu)$.
\item[\textit{(ii)}] The second variance-estimate \eqref{variance-estimate-2} has a completely different flavor, and that occasionally holds too, meaning that it can be seen as a \textit{Gamma characterization} estimate. By this we mean that the central assumption that the sequence $F_n$ converges in distribution towards the Gamma target distribution $G(\nu)$ is  heavily used to establish the estimate. 
\end{itemize}

In order to classify the convergence rate of a sequence, we introduce the following notation: When $(a_n)_{n\geq 1}$ and $(b_n)_{n\geq1}$ are two non-negative real number sequences, we write $a_n \approx_C b_n$ if $\lim_{n \to \infty} \frac{a_n}{b_n} = C$, for some constant $C>0$.


\begin{ex}\label{ex:toy-2}  Let $\alpha_n, \beta_n $ be two sequences of positive real numbers converging to zero as $n \to \infty$ and assume that $(1-\alpha_n)^2 + \beta^2_n = 1$ for all $n \ge 1$. Consider the following sequence in the second Wiener chaos 
\begin{align*}
F_n &=  c_{1,n} (N^2_1-1) + c_{2,n}(N^2_2-1) := (1 -\alpha_n)(N^2_1-1)    -  \beta_n (N^2_2-1)\\
&  \stackrel{\mathcal{D}}{\to} G(1), \quad \text{as } \, n \to \infty.
\end{align*}
Note that the second moment assumption $\E(F^2_n)= 2(1-\alpha_n)^2 + 2\beta^2_n = 2$ implies that $\beta_n \approx_C \sqrt{\alpha_n}$. Hence, after some straightforward computations, we arrive in the asymptotic relations
\begin{align*}
\Var\left( \Gamma_3(F_n) - 2 \Gamma_2(F_n)\right)& \approx_C (\alpha^2_n - \alpha_n)^2,\\
\Var\left( \Gamma_1(F_n) - 2 F_n \right) & \approx_C (\alpha_n - \sqrt{\alpha_n})^2.
\end{align*}
Also \vspace{-1em}
\begin{align*}
\Var  \left( \Gamma_2 (F) - 2 \Gamma_1 (F) \right) & = 2^5 \sum_{i=1}^{2} c^2_{i,n} (c^2_{i,n} - c_{i,n})^2\\
& \le 4 \Big( 2^3 \sum_{i=1}^{2} (c^2_{i,n} - c_{i,n})^2 \Big) = 4 \Var \left( \Gamma_1(F) - 2 F \right).
\end{align*}
Therefore, for some constant $C$ (independent of $n$), both estimates \eqref{variance-estimate-1},\eqref{variance-estimate-2} take place. Therefore, our main theorem \ref{thm:MainMalliavinSteinBound} yields that 
\[ 
d_2 (F_n, G(1)) \le_C \max \Big\{ \abs[\Big]{ \kappa_3 (F_n) - \kappa_3 (G(2)) }, \abs[\Big]{ \kappa_4 (F_n) - \kappa_4 (G(2)) } \Big\}. \]

\end{ex}

\begin{ex}\label{ex:toy-3}
In this example, instead, we consider the following sequence 
\begin{equation*}
\begin{split}
F_n & =  c_{1,n} (N^2_1-1) + c_{2,n}(N^2_2-1) := (1 -\alpha_n)(N^2_1-1)    + \beta_n (N^2_2-1) \\
& \stackrel{\text{law}}{\to} G(1), \quad \text{as } \, n \to \infty.
\end{split}
\end{equation*}
Similar computations as in Example \ref{ex:toy-2} yield that estimate \eqref{variance-estimate-1} is in order.  It is noteworthy that, as an alternative to the second estimate \eqref{variance-estimate-2}, the estimate 
\[ 
\Var\left( \Gamma_3(F_n) - 2 \Gamma_2(F_n)\right) \le_C \Big( \kappa_4 (F_n) - 6 \kappa_3(F_n) \Big)^2
\]
is also valid, which is enough for our purposes. Note that for the target random variable $G(1)$, we have $\kappa_4(G(1)) - 6 \kappa_3(G(1))=0$. Later on in Section \ref{sec:trace-class}, we will study this phenomenon in detail. Once again the square root in \eqref{eq:non-optimal-bound} can be improved. 

\end{ex}

\subsection{Iterated Gamma Operators: Variance Estimates}

\subsubsection{Variance Estimate: $\Var\left( \Gamma_2(F_n) - 2 \Gamma_1 (F_n) \right) \le_C \Var\left( \Gamma_1 (F_n) - 2 F_n\right)$}

We start with variance--estimate \eqref{variance-estimate-1}. We make use of a recent discovery in \cite{a-p-p} that the second Wiener chaos is stable under the Gamma operators, meaning that for any element $F$ in the second Wiener chaos, the resulting random variable $\Gamma_r (F)$ remains inside the second Wiener chaos up to a constant for any $r \ge 0$. 
\begin{lem}\label{lem:r+1<r}
Let $\nu >0$, and $F= I_2 (f)$ in the second Wiener chaos such that $\E[F^2]=2\nu$. Then, for every $r\ge 1$, with constant $C=4\nu$, we have 
\begin{equation}\label{eq:1}
\Var \left( \Gamma_{r+1} (F) - 2 \Gamma_r (F) \right) \le_C \Var \left( \Gamma_r (F) - 2 \Gamma_{r-1}(F) \right).
\end{equation}
In particular 
\[ 
\Var \left( \Gamma_{2} (F) - 2 \Gamma_1 (F) \right) \le (4\nu) \, \Var \left( \Gamma_1 (F) - 2 F \right).
\]
Also, for every $r \ge 1$, and with constant $ C=(4 \nu)^r $, we have the following variance-estimate
\begin{equation}\label{eq:2}
\Var \left( \Gamma_{r+1}(F) - 2 \Gamma_{r} (F) \right) \le_C \Var \left( \Gamma_1 (F) - 2 F \right).
\end{equation}
\end{lem}

\begin{proof}
Let's first prove estimate \eqref{eq:1}. Then estimate \eqref{eq:2} could be proven by iteration using similar arguments. 
%
Let $r \ge 1$. Denote by $A_f$ the associated Hilbert-Schmidt operator. As in the proof of Lemma \ref{lem:Var(Gamma_r-2Gamma_r-1)}, the variance of the random quantity $\Gamma_{r+1}(F) - 2 \Gamma_r (F)$ can be rewritten as 
\begin{align*}
\Var \left( \Gamma_{r+1} (F) - 2 \Gamma_r (F) \right) & = 2^{2r+3} \Tr \left(   (A^{r+2}_f  - A^{r+1}_f )^2 \right)\\
& =  2^{2r+3} \Tr \left(   A^2_f (A^{r+1}_f  - A^{r}_f )^2 \right)\\
& \le  2^{2r+3} \Tr (A^2_f)  \times \Tr \left(   (A^{r+1}_f  - A^{r}_f )^2 \right)\\
&= 2 ( 2  \Tr (A^2_f) ) \times  \Tr \left(   (A^{r+1}_f  - A^{r}_f )^2 \right)\\
&= 2 \kappa_2(F)  \times \Tr \left(   (A^{r+1}_f  - A^{r}_f )^2 \right) \\
&= 4 \nu \, \Var \left( \Gamma_{r} (F) - 2 \Gamma_{r-1} (F) \right),
\end{align*}
where in the third step, we have used the following trace inequality for non-negative operators (see \cite{Liu}),
\[ \Tr (AB) \le \Tr(A) \, \Tr(B) \qquad \text{for} \quad  A ,B \ge 0. \]
\end{proof}

\begin{rem}\label{rem:Gamma-2-cumulant}  A direct consequence of Lemma \ref{lem:r+1<r} is, that for a random element $F$ in the second Wiener chaos with $\Var \left( \Gamma_1 (F) - 2 F \right)=0$ (and therefore $F = G (\nu)$ in distribution), we necessarily obtain for $ r \ge 2$,
\begin{equation}\label{eq:3}
\begin{split}
0 &= \Var \left( \Gamma_{r+1}(F) - 2 \Gamma_{r} (F) \right) \\
&= \frac{1}{(2r+3)!} \kappa_{2r+4}(F) -  \frac{4}{(2r+2)!} \kappa_{2r+3}(F) +  \frac{4}{(2r+1)!} \kappa_{2r+2}(F).
\end{split}
\end{equation} 
Later on in Section \ref{sec:gamma-charac}, we will show that astonishingly the converse is also true. Precisely, for the random element $F$ in the second Wiener chaos with $\E(F^2)=2\nu$, the sole assumption $ \Var \left( \Gamma_{r+1}(F) - 2 \Gamma_{r} (F) \right)=0$ for \textbf{some} $r \ge 2$, implies that $F$ necessarily is Gamma distributed.  
\end{rem}


\subsubsection{Variance Estimate: $\Var \left( \Gamma_3 (F_n) - 2 \Gamma_2(F_n) \right)  \le_C  \Var^2 \left( \Gamma_1 (F_n) - 2 F_n \right)$}\label{sec:2variance-estimate}

We begin with the following important observation, namely that a sequence in the second Wiener chaos can only converge to a centered chi-squared distribution $\chi^2$, not to any other centered Gamma distribution.

\begin{Prop} \label{prop:OnlyIntegerNuPossible}
	Let $\nu>0$, and let $( F_n = I_2 (f_n) )_{n \geq 1}$ be a sequence of random variables in the second Wiener chaos with fixed variance $\E[F_n^2]=2 \nu$ for all $n \geq 1$. 
	Denote by $c_{j,n}$ the $j$-th eigenvalue of the Hilbert-Schmidt operator $A_{f_n}$ associated with $F_n$. Without loss of generality, assume that $ \abs{c_{1,n}} \geq \abs{c_{2,n}} \geq \ldots$ (see Assumption \ref{assum:OrderedEigenvalues}). Then $F_n$ converges in distribution to $G(\nu) \sim CenteredGamma(\nu)$ if and only if
	\begin{enumerate}
		\item[(a)] $\nu$ is an integer, and
		\item[(b)] $c_{j,n} \stackrel{n \to \infty}{\longrightarrow} 1$ for $j=1,\ldots, \nu$ \quad and \quad $c_{j,n} \stackrel{n \to \infty}{\longrightarrow} 0$ for $j> \nu$.
		\end{enumerate}
\end{Prop}

\begin{proof}
	Assume that $F_n \stackrel{\mathcal{D}}{\to} G(\nu)$ for some $\nu > 0$. Since this implies convergence of all cumulants, \eqref{eq:SecondWienerChaosCumulantFormulaEigenvalues} and \eqref{eq:CenteredGammaCumulants} imply that
	\begin{align}
	 & 2^{p-1} (p-1)! \sum_{i=1}^{\infty} c_{i,n}^p \to 2^{p-1} (p-1)! \, \nu  \quad \text{as } n \to \infty \notag \\
	 \Leftrightarrow \quad & \sum_{i=1}^{\infty} c_{i,n}^p  \to \nu \quad \text{as } n \to \infty, \label{eq:ProofOnlyIntegerNuPossible1}
	 \end{align}
	 for all $p \geq 2$. Furthermore,  $F_n \stackrel{\mathcal{D}}{\to} G(\nu)$  implies $\Var( \Gamma_1(F_n) - 2 F_n) \to 0$, see e.g. \cite{n-p-noncentral}, Theorem 1.2 condition (v). Hence by Lemma \ref{lem:Var(Gamma_r-2Gamma_r-1)}, we have, for all $j \in \N$, that
	 \[ c_{j,n}^2 (c_{j,n} - 1)^2 \leq  \sum_{i=1}^{\infty} c_{i,n}^2 ( c_{i,n} - 1)^2 = \frac{1}{2^3} \Var\big( \Gamma_1(F_n) - 2 F_n \big) \to 0 \quad \text{as } n \to \infty. \]
	 From this we deduce that for all $j$, the sequence $(c_{j,n})_{n\geq 1}$ is bounded and can only have accumulation points $0$ and $1$.
	 
	 First, consider $(c_{1,n})_{n\geq 1}$. Assume there exists a subsequence $(c_{1,n_k})_{k \geq 1}$ that converges to $0$. Then using the ordering of the eigenvalues, we get
	 \[ \sum_{i=1}^{\infty} c_{i,n_k}^4 \leq c_{1,n_k}^2 \sum_{i=1}^{\infty} c_{i,n_k}^2 \stackrel{k \to \infty}{\longrightarrow} 0 \times \nu = 0, \]
	 which contradicts \eqref{eq:ProofOnlyIntegerNuPossible1}. Hence $\lim_{n \to \infty} c_{1,n} = 1$. What remains is
	 \begin{equation}
	 \lim_{n \to \infty} \sum_{i=2}^{\infty} c_{i,n}^p = \nu - 1 \qquad \text{for all } p \geq 2. \label{eq:ProofOnlyIntegerNuPossible2}
	 \end{equation}
	 From here, we continue inductively, each time subtracting $1$ from the right hand side of \eqref{eq:ProofOnlyIntegerNuPossible2}. Since the right hand side cannot bet negative, we conclude that $\nu$ must be an integer and that $\lim_{n\to\infty}c_{1,n} = \ldots = \lim_{n\to\infty}c_{\nu,n} = 1$.
	 
	 Now we are left with
	 \begin{equation}
	 \lim_{n \to \infty} \sum_{i=\nu +1}^{\infty} c_{i,n}^p = 0 \qquad \text{for all } p \geq 2, \label{eq:ProofOnlyIntegerNuPossible3}
	 \end{equation}
	 from which we deduce for all $j \geq \nu+1$ that
	 \[ c_{j,n}^2 \leq \sum_{i=\nu +1}^{\infty} c_{i,n}^2 \stackrel{n \to \infty}{\longrightarrow} 0, \]
	 and thus $\lim_{n\to\infty} c_{j,n}=0$.\\
	 \\
	 Conversely, suppose that $\nu$ is an integer and $(b)$ holds. The target has the representation $G(\nu) = \sum_{i=1}^{n} (N_i^2-1)$. Therefore, the $L^2$-distance between $F_n$ and $G(\nu)$ is given by
	 \[ \E\big[ \big(F_n - G(\nu) \big)^2\big] = 2 \sum_{i=1}^{\nu} (c_{i,n}-1)^2 + 2 \sum_{i=\nu+1}^{\infty} c_{i,n}^2. \]
	 The first term goes to zero since there are only finitely many summands. For the second term, the assumption $\E[F_n^2] = 2 \nu$ yields
	 \[ \sum_{i=1}^{\infty} c_{i,n}^2 = \nu \quad \text{for all }n \quad \implies \quad \lim\limits_{n \to \infty} \sum_{i=\nu+1}^{\infty} c_{i,n}^2 = 0. \]
	 Hence $F_n \to G(\nu)$ in $L^2$.
	 \end{proof}

	\begin{rem}
	\begin{enumerate}
		\item[\textit{(i)}]
		For the implication $F_n \stackrel{\mathcal{D}}{\to} G(\nu) \implies (a)$ and $(b)$, we can drop the assumption $\Var(F_n)=2\nu$, but not the ordering of the eigenvalues. Take for example the sequence
		\[ F_n = \sum_{i=1}^{\infty} c_{i,n} ( N_i^2 -1 ) =  N_n^2-1, \]
		i.e. $c_{i,n} = \ind{i=n}$. Then obviously $F_n \stackrel{\mathcal{D}}{=} G(1)$ for all $n$, but $\lim_{n\to\infty} c_{i,n} = 0$ for all $i \in \N$.
		\item[\textit{(ii)}] For the converse, $(a)$ and $(b) \implies F_n \stackrel{\mathcal{D}}{\to} G(\nu)$, we do not need to order the eigenvalues (descending absolute value), but cannot drop the assumption $\Var(F_n)=2\nu$. Take for example the sequence
		\[ F_n = (N_1^2 - 1) + \sum_{i=2}^{n+1} \frac{1}{\sqrt{2n}} (N_i^2 -1) := (N_1^2 - 1) + S_n. \]
		The sum $S_n$ is independent of $N_1^2 -1$, and thus by the central limit theorem
		$F_n \stackrel{\mathcal{D}}{\to} G(1) + N$, where $G(1) \sim CenteredGamma(1)$ and $N$ is an independent $\mathscr{N}(0,1)$ variable. Here $c_{1,n} \to 1$ and $c_{i,n} \to 0$ for $i \geq 2$, so $\nu=1$. However, for all $n$, we have $\Var(F_n)=3 \neq 2\nu$.
	\end{enumerate}
	\end{rem}

Because of Proposition \ref{prop:OnlyIntegerNuPossible}, from now on, we will only focus on cases where $\nu$ is an integer. Also recall that on second Wiener chaos $\Gamma_j = \Gamma_{alt,j}$ for all $j$, so we will always use the notation without the additional subscript. Unlike the variance estimate \eqref{variance-estimate-1}, in order to keep transparency in analyzing the validity of the second variance estimate \eqref{variance-estimate-2}, we discuss the following different cases separately.

\begin{Prop}\label{lem:finite-eigenvalue-1}(The case of finitely many eigenvalues)
Let $\nu >0$ and $M \geq 2$. For each $n \geq 1$, let $c_{1,n}, \ldots, c_{M,n}$ be ordered by descending absolute value (see Assumption \ref{assum:OrderedEigenvalues}) and assume that $\sum_{i=1}^{M}c^2_{i,n} = \nu$. Furthermore, assume that as $n \to \infty$,
\begin{align*}
F_n :&= c_{1,n} (N^2_1 -1) + c_{2,n} (N^2_2 -1)+ \cdots + c_{M,n}(N^2_M -1) \\
& \stackrel{\mathcal{D}}{\to} G (\nu),
\end{align*}
where $G(\nu)$ is a centered Gamma random variable, and $\{N_i \}_{1 \le i \le M}$ is a family of independent $\mathscr{N}(0,1)$ random variables. Then $\nu \in \{ 1,2,\cdots,M\}$ is an integer, and therefore the target $G(\nu)$ is a centered $\chi^2$ random variable with $\nu$ degrees of freedom. Set 
\begin{equation}\label{eq:rates}
\omega(n):=  \max \big\{ \abs{ 1 - c_{i,n} } \, : \, i \in \{1,\ldots, \nu \} \big\}, \quad \text{ and} \quad \vartheta(n):= \sum_{i=\nu+1}^{M} c^2_{i,n}.
\end{equation}
Then $F_n \stackrel{L^2}{\to} G(\nu)$, as $n \to \infty$, and the rate of the convergence in the square mean is $\max\{ \omega(n)^2, \vartheta(n)\}$. Furthermore,
\begin{itemize}
\item[(a)] the asymptotic assertion
\[
\Var\left( \Gamma_3(F_n) - 2 \Gamma_2(F_n)\right)  \approx_C \Var^{\, 2} \left( \Gamma_1(F_n) - 2 F_n \right)
\]
holds if and only if $\vartheta(n) \approx_C \omega(n)$. Also the latter asymptotic relation verifies whenever the degree of freedom $\nu =1$.
\item[(b)] the asymptotic assertion
\[ \Var\left( \Gamma_3(F_n) - 2 \Gamma_2(F_n)\right)  \approx_C \Var \left( \Gamma_1(F_n) - 2 F_n \right)
\]
holds if and only if, the degree of freedom $\nu = M$ is the largest possible value, \textbf{or} $\vartheta(n) \approx_C \omega(n)^2$, \textbf{or} $\vartheta(n)=o \left( \omega(n)^2\right)$.
\end{itemize}
\end{Prop}
\begin{proof}
The first part follows immediately from Proposition \ref{prop:OnlyIntegerNuPossible}. Now, the second moment assumption $\E(F^2_n) = 2 \sum_{1 \le i \le M} c^2_{i,n} =2 \nu$ implies that
\begin{align*}
\sum_{i=1}^{\nu} (1 - c_{i,n})^2 &= \sum_{i =1}^{\nu} ( 1 + c^2_{i,n} - 2 c_{i,n}) = \sum_{i=1}^{\nu} ( 1 + c^2_{i,n}) + \sum_{i = \nu+1}^{M} c^2_{i,n} - 2 \sum_{i=1}^{\nu} c_{i,n} - \sum_{i =\nu+1}^{M} c^2_{i,n} \\
&= 2 \sum_{i=1}^{\nu} (1 - c_{i,n}) - \sum_{i=\nu+1}^{M} c^2_{i,n}.
\end{align*}
Therefore,
\begin{equation}\label{eq:L2-rate}
\E(F_n - G(\nu))^2 =  \sum_{i=1}^{\nu} (1 - c_{i,n})^2 + \sum_{i=\nu+1}^{M} c^2_{i,n} = 2 \sum_{i=1}^{\nu} (1 - c_{i,n}) \to 0.
\end{equation}
\textit{Proof of $(a)$} : when $\nu =M$, then for all $1 \le i \le M$, the coefficients $c_{i,n} \to 1$, as $n \to \infty$. Hence,
\begin{align*}
\Var\left( \Gamma_3(F_n) - 2 \Gamma_2(F_n)\right)& = 2^7\sum_{i=1}^{M} c^6_{i,n} (1 - c_{i,n})^2  \approx \Big\{ \max \{ \abs{ 1 - c_{i,n} } \, : \, i=1,\cdots,M\} \Big \}^2, \\
  \Var \left( \Gamma_1(F_n) - 2 F_n \right) &=   2^3 \sum_{i=1}^{M} c^2_{i,n} (1- c_{i,n})^2 \approx  \Big\{ \max \{ \abs{ 1 - c_{i,n} } \, : \, i=1,\cdots,M\} \Big \}^2. 
\end{align*} 
Therefore
\[
\Var\left( \Gamma_3(F_n) - 2 \Gamma_2(F_n)\right) \approx_C  \Var \left( \Gamma_1(F_n) - 2 F_n \right).
\]
Hence, we assume that $\nu < M$. Then $c_{1,n}, \ldots, c_{\nu,n} \to 1$, and $c_{\nu+1,n}, \ldots, c_{M,n} \to 0$ as $n \to \infty$. Since $\{ \nu+1, \ldots, M \}$ is finite, and $\vartheta(n) \le 2 \nu \, \omega(n)$, we have $ \sum_{i = \nu+1}^{M} c^6_{i,n} = o \left( \omega(n) ^2 \right)$. Hence


\begin{align*}
& \Var\left( \Gamma_3(F_n) - 2 \Gamma_2(F_n)\right) = 2^7\sum_{i=1}^{M} c^6_{i,n} (1 - c_{i,n})^2 \\
& \qquad = 2^7 \Big\{ \sum_{i =1}^{\nu} c^6_{i,n} (1 - c_{i,n})^2 + \sum_{i=\nu+1}^{M} c^6_{i,n} (1 - c_{i,n})^2 \Big\}  \approx_C \max \Big\{ \omega(n)^2, o \big( \omega(n) ^2 \big) \Big\} \approx_C \omega(n)^2.
\end{align*}
Also,
\begin{align}
  \Var \left( \Gamma_1(F_n) - 2 F_n \right) & =  2^3 \Big\{ \sum_{i=1}^{\nu} c^2_{i,n} (1- c_{i,n})^2  + \sum_{i=\nu+1}^M c^2_{i,n} (1- c_{i,n})^2 \Big \} \notag \\
  &  \approx_C \max \{  \omega(n)^2, \vartheta(n)\}. \label{eq:1-0}
\end{align}
Hence,
\begin{equation*}
  \Var\left( \Gamma_3(F_n) - 2 \Gamma_2(F_n)\right) \approx_C \Var^2 \left( \Gamma_1(F_n) - 2 F_n\right) \quad \text{if and only if} \quad \vartheta(n) \approx_C \omega(n).
\end{equation*}
In addition, assumption $\omega(n)  \approx_C  \vartheta(n)$ is equivalent to 
\begin{equation}\label{eq:nu=1}
\E(F_n - G(\nu))^2 = 2 \sum_{i =1}^{\nu} (1 - c_{i,n}) \approx_C \omega(n)  \approx_C \sqrt{\sum_{i =1}^{\nu} (1 - c_{i,n})^2}.
\end{equation}
Therefore, when the degree of freedom $\nu=1$, \eqref{eq:nu=1} occurs. \textit{Proof of $(b)$}: It can be discussed in a similar way.
\end{proof}

\begin{rem}\label{rem:1} In the light of relation \eqref{eq:L2-rate}, always $\vartheta(n) \le 2 \nu \, \omega(n)$. Taking this into account together with 
\begin{align*}
& \Var\left( \Gamma_2(F_n) - 2 \Gamma_1(F_n)\right) = 2^5\sum_{i=1}^{M} c^4_{i,n} (1 - c_{i,n})^2 \\
& \qquad = 2^5 \Big\{ \sum_{i =1}^{\nu} c^4_{i,n} (1 - c_{i,n})^2 + \sum_{i =\nu+1}^{M} c^4_{i,n} (1 - c_{i,n})^2 \Big\} \approx_C \max \{  \omega(n)^2, \omega(n)^2\} \approx_C \omega(n)^2,
\end{align*}
one can conclude that the asymptotic estimate
\[ \Var\left( \Gamma_2(F_n) - 2 \Gamma_1(F_n)\right) \approx_C  \Var\left( \Gamma_3(F_n) - 2 \Gamma_2(F_n)\right)
\]
takes place as soon as the sequence $F_n$ in the second Wiener chaos converges in distribution towards the centered Gamma distribution $G(\nu)$ without any further assumptions. 
\end{rem}

\begin{ex}
The following simple example shows that, in general, many things can happen. Let $\delta \in [0,1]$, and  consider the sequence $F_n= \sum_{i=1}^{5} c_{i,n} (N^2_i  -1)$ in the second Wiener chaos, where the coefficients $c_{i,n}$ are given as   
 \begin{align*}
 c_{1,n} &=\sqrt{1+\frac{1}{n}}, \quad c_{2,n}=\sqrt{1- \frac{1}{n}},\\
 c_{3,n} & = \sqrt{1 - \frac{1}{n^{1+\delta}}}, \quad c_{4,n}= \sqrt{\frac{1}{2 n^{1+\delta}}}, \quad \text{and} \quad c_{5,n}=\sqrt{\frac{1}{2 n^ {1+\delta}}}.
\end{align*}
Then,
\[
c^2_{4,n} \approx_C \frac{1}{n^{1+\delta}} \approx_C  c^2_{5,n} \approx_C \vartheta(n) \approx_C \omega(n)^{1+\delta}.
\]
Therefore, when $\delta =0$, then our favorite estimate 
\[
\Var\left( \Gamma_3(F_n) - 2 \Gamma_2(F_n)\right) \approx_C   \Var\, ^2 \left( \Gamma_1(F_n) - 2 F_n \right)
\]
takes place, and when $\delta=1$, then 
\[
\Var\left( \Gamma_2(F_n) - 2 \Gamma_1(F_n)\right) \approx_C \Var\left( \Gamma_3(F_n) - 2 \Gamma_2(F_n)\right) \approx_C   \Var \left( \Gamma_1(F_n) - 2 F_n \right).$$ In general $$  \Var\left( \Gamma_3(F_n) - 2 \Gamma_2(F_n)\right) \approx_C \Big( \Var \left( \Gamma_1(F_n) - 2 F_n \right) \Big)^{\frac{2}{1+\delta}}.
\]
One can also consider more involved intermediate rates such as $\vartheta(n) \approx_C \omega(n)^{1+\delta} \, \log^{\gamma} (\omega(n))$ for some $\delta, \gamma \ge 0$.

\end{ex}

\begin{cor}\label{cor:variance123-1}
Let $M \ge 2$ and $\nu >0$. Consider a sequence $(F_n)_{n \ge 1}$ of random elements in the second Wiener chaos such that $\E(F^2_n)=2\nu$ for all $ n\ge 1$, possessing the representation 
\[ F_n = \sum_{1 \le i \le M} c_{i,n} (N^2_i -1), \qquad n \ge 1,
\]
with $\abs{c_{1,n}} \geq \ldots \geq \abs{c_{M,n}}$ for each $n \geq 1$ (see Assumption \ref{assum:OrderedEigenvalues}).
Also, we assume that $F_n$ converges in distribution towards a centered Gamma distribution with parameter $\nu >0$. Then there exist two constants $0< C_1 <C_2$ (may depend on sequence $F_n$, but independent of $n$), such that for all $n \ge 1$,
\begin{itemize}
\item[(i)] if $\nu=1$, \textbf{or} $\vartheta(n) \approx_C \omega(n)$, then $$C_1  \Var \, ^2 \left( \Gamma_1(F_n) - 2 F_n \right) \le \Var\left( \Gamma_3(F_n) - 2 \Gamma_2(F_n)\right) \le C_2  \Var \,^2 \left( \Gamma_1(F_n) - 2 F_n \right).$$
\item[(ii)] if $\nu =M$, \textbf{or} $\vartheta(n) \approx_C \omega(n)^2$,  \textbf{or} $\vartheta(n)=o \left( \omega(n)^2\right)$, then  $$C_1  \Var\left( \Gamma_1(F_n) - 2 F_n \right) \le \Var\left( \Gamma_3(F_n) - 2 \Gamma_2(F_n)\right) \le C_2   \Var \left( \Gamma_1(F_n) - 2 F_n \right).$$
\end{itemize}
\end{cor}

\begin{rem}(Case $\nu=M$)  Let $M \ge 2$ and $\nu >0$. Assume that $\E(F^2_n) = 2 \nu$ for all $n \ge 1$ where 
$$F_n=\sum_{1 \le i \le M} c_{i,n} (N^2_i -1) \stackrel{\mathcal{D}}{\to} G(\nu=M), \quad \text{ as } n \to \infty.$$  The second moment assumption implies that $  \sum_{i=1}^{M} (1-c_{i,n})^2 = 2  \sum_{i=1}^{M}  (1 - c_{i,n}) \ge 0$. On the other hand (up to some constants), 
\begin{equation}\label{eq:3-cumulant}
\begin{split}
\abs[\Big]{ \kappa_3 (F_n) - \kappa_3 (G(\nu)) } &= \abs[\Big]{ \sum_{i=1}^{M}( c^3_{i,n} -1) } = \abs[\Big]{ \sum_{i=1}^{M} (c_{i,n} -1) (c^2_{i,n} + c_{i,n} +1) } \\
& = \abs[\Big]{ \sum_{i=1}^{M} (c_{i,n} -1) \left( (c_{i,n} -1)^2 + 3 c_{i,n} \right) } \\
&= \abs[\Big]{ \sum_{i=1}^{M} (c_{i,n} -1) \Big( (c_{i,n} -1)^2 + 3 (c_{i,n} -1) +3 \Big) } \\
&= \abs[\Big]{ 3 \sum_{i=1}^{M} (c_{i,n} -1) + 3   \sum_{i=1}^{M} (c_{i,n} -1)^2 +  \sum_{i=1}^{M}  (c_{i,n} -1)^3 } \\
&=  \abs[\Big]{ 9  \sum_{i=1}^{M} (c_{i,n} -1) + \sum_{i=1}^{M}  (c_{i,n} -1)^3 } \\
& \approx_C \abs[\Big]{  \sum_{i=1}^{M} (c_{i,n} -1) },
\end{split}
\end{equation}
which in general is less than the rate $ \max \Big\{ \abs{ 1 - c_{i,n} } \, : \, i=1,\cdots,M \Big\}$. Similarly, 
\begin{equation}\label{eq:4-cumulant}
\abs[\Big]{ \kappa_4 (F_n) - \kappa_4 (G(\nu)) } \approx_C   \abs[\Big]{ \sum_{i=1}^{M} (c_{i,n} -1) }.
\end{equation} 
Hence, the following remarks of independent interest are in order.
\begin{itemize}
\item[\textit{(i)}] Observations \eqref{eq:3-cumulant} and \eqref{eq:4-cumulant} reveal that either of the sole moment convergences $\E(F^3_n) \to \E(G(\nu)^3)$ or $\E(F^4_n) \to \E(G(\nu)^4)$ implies convergence in distribution of the sequence $F_n$ towards the target distribution $G(\nu)$. In other words, the third moment criterion implies the fourth moment criterion and vice versa. Such phenomenon has been already observed in the case of normal approximation, see \cite{viens}. 
\item[\textit{(ii)}] It is worth mentioning that if $M=\nu \ge 5$, then \cite[Theorem 1.2]{rola} yields that 
in fact, in the stronger distance $d_{TV}$, there exists a constant $C$ (may depends on sequence $F_n$, but independent of $n$), such that for all $n\ge 1$,
\begin{equation*}
d_{TV} (F_n, G(\nu)) \le_C  \max \Big\{ \abs{ 1 - c_{i,n} } \, : \, i=1,\cdots,M \Big\}.
\end{equation*}
Hence,
\begin{equation}\label{eq:TV-rate}
d_{TV} (F_n, G(\nu)) \le_C \sqrt{ \max \Big\{ \abs[\Big]{ \kappa_3 (F_n) - \kappa_3 (G(\nu)) }, \abs[\Big]{ \kappa_4 (F_n) - \kappa_4 (G(\nu)) } \Big\}}.
\end{equation}
 We conjecture that in this setting, the estimate \eqref{eq:TV-rate} continues to hold when removing the assumption $\nu \ge 5$. See also Proposition \ref{Prop:concrete-example} in Section \ref{sec:concrete-example}., and Conjecture \ref{con:<>}.
\end{itemize}

\end{rem}

\begin{Prop}\label{lem:infinite-eigenvalues} (The case of ultimately infinitely many non-zero eigenvalues)
Let $\nu >0$, and $( M_n )_{ n \ge 1} \subset \N \cup \{+ \infty\}$ be a sequence such that $M_n \uparrow \infty$. Consider a sequence $(F_n)_{n \ge 1}$ of random elements in the second Wiener chaos such that $\E[F^2_n]=2\nu$ for all $ n\ge 1$, possessing the following representation 
\[
F_n = \sum_{1 \le i \le M_n} c_{i,n} (N^2_i -1), \qquad n \ge 1,
\]
where for each $n \geq 1$, it holds that $\abs{c_{1,n}} \geq \ldots \geq \abs{c_{M_n,n}}$ (see Assumption \ref{assum:OrderedEigenvalues}).
Also, we assume that $F_n$ converges in distribution towards a centered Gamma distribution $G(\nu)$ with parameter $\nu >0$. Then, the asymptotic relation 
\[
\Var\left( \Gamma_3(F_n) - 2 \Gamma_2(F_n)\right)  \approx_C \Var^{\,2} \left( \Gamma_1(F_n) - 2 F_n \right)
\]
holds if and only if $\vartheta(n) \approx_C \omega(n)$. Consequently, whenever the aforementioned asymptotic condition takes place, there exist two constants $0< C_1 <C_2$ (may depend on sequence $F_n$, but independent of $n$) such that for all $n \ge 1$,
\[
C_1  \Var \, ^2 \left( \Gamma_1(F_n) - 2 F_n \right) \le \Var\left( \Gamma_3(F_n) - 2 \Gamma_2(F_n)\right) \le C_2  \Var \,^2 \left( \Gamma_1(F_n) - 2 F_n \right).
\]
\end{Prop}

\begin{proof}
First note that since $M_n \uparrow \infty$, we have $M_n > \nu$ for large enough values of $n$. So without loss of generality, we assume $M_n = \infty$ for all $n\ge 1$. Using Proposition \ref{prop:OnlyIntegerNuPossible} we deduce that $\nu$ is an integer, and that as $n \to \infty$, $c_{1,n},\ldots, c_{\nu,n} \to 1$ and $c_{i,n} \to 0$ for all $i \geq \nu+1$.
Then relation \eqref{eq:1-0} yields that 
\[
\Var \left( \Gamma_1(F_n) - 2 F_n \right) \approx_C \omega(n)
\]
if and only if  $\vartheta(n) \approx_C \omega(n)$, where as before $\omega(n)= \max  \{ \abs{ 1 - c_{i,n} } : i = 1, \ldots, \nu \}$. Note that there are infinitely many coefficients tending to zero. We claim that 
\[
\sum_{i \geq \nu+1 }c^6_{i,n} = o (\omega(n)^2).
\]
To this end, take a nested sequence $A_1 \subseteq A_2 \subseteq \ldots \subseteq A_m \subseteq A_{m+1} \subseteq \ldots $ such that $A_m \to \{ \nu+1, \nu+2, \ldots \}$ as $m \to \infty$, and $\# A_m < \infty$ for all $m \ge 1$.  Define
\[
x_{m,r}(n):= \sum_{i \in A_m} c^r_{i,n}.
\]
Then for each $m \in \N$, the estimate $x_{m,2}(n) \le \sum_{i \geq \nu+1 }c^2_{i,n} = \vartheta(n) \le 2\nu \, \omega(n)$ holds.  So the above analysis, together with the fact that $\# A_m$ is finite for $m \ge 1$, tells us that
\[
x_{m,6}(n) = o(\omega(n)^2), \quad \forall \, m\ge 1.
\]
Now, taking into account that $x_{m,6} \to x_{\infty,6}(n):= \sum_{i \geq \nu +1 }c^6_{i,n} $, as $m \to \infty$, and each $x_{m,6}(n) = o(\omega(n)^2)$, a direct application of the monotone convergence theorem implies that $x_{\infty,6}= o(\omega(n)^2)$. Therefore,  
\begin{align*}
& \Var \big( \Gamma_3(F_n) - 2 \Gamma_2(F_n) \big) = 2^7\sum_{i=1}^{\infty} c^6_{i,n} (1 - c_{i,n})^2 \\
& \qquad = 2^7 \Big\{ \sum_{i=1}^{\nu} c^6_{i,n} (1 - c_{i,n})^2 + \sum_{i=\nu+1}^{\infty} c^6_{i,n} (1 - c_{i,n})^2 \Big\} \approx_C \max \{ \abs{ 1 - c_{i,n} }^2\, : \, i = 1,\ldots, \nu \} \approx_C \omega(n)^2.
\end{align*}
Hence the claim follows. 
\end{proof}

\subsection{An Optimal Theorem}\label{sec:optimal}
Now we are ready to state our main theorem providing an optimal rate of convergence in terms of the third and the fourth cumulants.  The following result  provides an analogous counterpart to the same phenomenon in the case of normal approximation on arbitrary Wiener chaos, see \cite[Theorem 1.2]{n-p-optimal}. 
\begin{Thm}\label{thm:optimal}
Let $\nu >0$. Assume that
\[
( F_n )_{n \ge 1}= \Big( \sum_{ i \ge 1} c_{i,n} (N^2_i -1) \Big)_{n \ge 1}
\]
is a sequence of elements in the second Wiener chaos such that $\abs{c_{1,n}} \geq \abs{c_{2,n}} \geq \ldots$ (see Assumption \ref{assum:OrderedEigenvalues}) and $\E(F^2_n)= 2 \sum_{ i \ge 1} c^2_{i,n} = 2\nu$ for all $n \ge 1$. Assume, in addition,  as $n \to \infty$, that 
\begin{equation}\label{eq:3-4-cumulants}
\Var \big( \Gamma_1(F_n) - 2 F_n \big) \to 0.
\end{equation}
Then $F_n$ converges in distribution towards a centered Gamma distribution $G(\nu)$ with parameter $\nu$. Furthermore, when $\vartheta(n) \approx_C \omega(n)$, where $\vartheta(n)$ and $\omega(n)$ are as in \eqref{eq:rates}, then there exist two constants $0 < C_1 < C_2$ (possibly depending on the sequence $F_n$, but independent of $n$) such that for all $n \ge 1$,
\begin{equation}\label{eq:optimal-finite-eigenvalue}
C_1 \,  \mathbf{M} (F_n) \le d_{2} (F_n, G(\nu)) \le C_2 \,  \mathbf{M} (F_n),
\end{equation}
where as before
\[
\mathbf{M}(F_n) :=  \max \Big\{ \abs[\Big]{ \kappa_3 (F_n) - \kappa_3 (G(\nu)) }, \abs[\Big]{ \kappa_4 (F_n) - \kappa_4 (G(\nu)) }  \Big\}.
\]
\end{Thm}

\begin{proof}
The asymptotic relation \eqref{eq:3-4-cumulants} implies that $F_n$ converges in distribution towards a centered Gamma distribution $G(\nu)$, which is a well known fact, see for example \cite{n-p-noncentral} .\\
(\textit{upper bound}): This is a direct application of Theorem \ref{thm:MainMalliavinSteinBound}, Corollary \ref{cor:variance123-1}, and Proposition \ref{lem:infinite-eigenvalues}. (\textit{lower bound}): Fix a real number $\rho >0$ whose range of values will be determined later on. Taking into account the second moment assumptions, it is a classical result (see \cite[Chapter $7$]{lukas}) that the characteristic functions $\phi_{F_n}$ and $\phi_{G(\nu)}$ are analytic inside the strip $\Delta_\nu := \{ z \in \mathbb{C} :  \abs{ \operatorname{Im}z } < \frac{1}{2\sqrt{\nu}} \}$. Moreover, in the strip of regularity $\Delta_\nu$, they follow the integral representations
\[
\phi_{F_n}(z) = \int_\R e^{izx} \mu_{n}(dx) \quad \text{and} \quad \phi_{G(\nu)}(z) = \int_\R e^{izx} \mu_\nu (dx),
\]
where $\mu_n$ and $\mu_\nu$ stand for the probability measures of $F_n$ and $G(\nu)$ respectively. Recall that all elements in the second Wiener chaos have exponential moments, see \cite[Proposition $2.7.13$, item (iii)]{n-pe-1}. Denote by $\Omega_{\rho, \nu}$ the domain
\[\Omega_{\rho, \nu}: = \Big\{ z=t + i y \in \mathbb{C} \, : \, \abs{ \operatorname{Re}z } < \rho, \abs{ \operatorname{Im}z } < \min \{ (2\sqrt{\nu})^{-1}, e^{-1}\} \Big\}.
\]
Then for any $z \in \Omega_{\rho,\nu}$, together with a Fubini's argument, we have that
\begin{align*}
\abs[\Big]{ \phi_{F_n}(z) - \phi_{G(\nu)}(z) } & = \abs[\Big]{ \int_\R  e^{itx - y x} (\mu_n - \mu_\nu) (dx) } = \abs[\Big]{ \sum_{ k \ge 0} \frac{(-y)^k}{k!} \int_\R x^k e^{itx} (\mu_n - \mu_\nu)(dx) } \\
& \le \sum_{k\ge 0}\frac{e^{-k}}{k!} \abs[\Big]{ \phi^{(k)}_{F_n}(t)- \phi^{(k)}_{G(\nu)} (t) }  \le  \sum_{k\ge 0}\frac{e^{-k}}{k!} \rho^{k+1} d_2 (F_n, G(\nu)) \\
& = \rho \, e^{\rho e^{-1}} d_2 (F_n, G(\nu)).
\end{align*}
 Hence $ \abs{ \phi_{F_n}(z)  - \phi_{G(\nu)}(z) } \le_{C_{\rho}} d_2 (F_n, G(\nu))$ 
  for every $z \in \Omega_{\rho,\nu}$. Let $R>0$ such that the disk $D_R \subset \mathbb{C}$ with the origin as center and radius $R$ is contained in the domain $\Omega_{\rho,\nu}$ (note that $R$ depends only on $\nu$, since $\rho$ is a free parameter. For example, one can choose $ \min \{ (2\sqrt{\nu})^{-1}, e^{-1}\} < \rho <2  \min \{ (2\sqrt{\nu})^{-1}, e^{-1}\}$). Now for any $z \in D_R$, and using the fact that
 \[
 \frac{1}{\phi^2_{G(\nu)}(z)}= \left(e^{2iz} (1-2iz) \right)^\nu,
 \]
 one can readily conclude that the function $\phi_{G(\nu)}(z)$ is bounded away from $0$ on the disk $D_R$. Also, for any $r \ge 2$,
 \begin{equation}\label{eq:lower-2}
 \begin{split}
\abs[\big]{ \kappa_r (F_n) } & \le 2^{r-1}(r-1)! \sum_{i\ge 1} \abs{ c_{i,n} }^{r}  \le  2^{r-1}(r-1)! \max_{i} \abs{ c_{i,n} }^{r-2} \sum_{i\ge 1} \abs{ c_{i,n} }^{2}\\
& \le  2^{r-2}(r-1)! \sqrt{\nu}^{\,r-2} \, \E(F^2_n) = 2^{r-2}(r-1)! \sqrt{\nu}^{\,r}.
\end{split}
\end{equation}
Therefore, for any $z \in D_R$,
\begin{align*}
\abs[\Big]{ \frac{1}{\phi_{F_n}(z)} } \le \exp \Big\{ \sum_{r \ge 2} \frac{ \abs{ \kappa_r(F_n) } }{r!} \abs{z}^r \Big\} & \le  \exp \Big\{ \sum_{r \ge 2} \frac{2^{r-2}(r-1)! \sqrt{\nu}^{\,r}}{r!} \abs{z}^r \Big\}\\
& \le  \exp \Big\{ \sum_{r \ge 2} \frac{2^{r-2}(r-1)! \sqrt{\nu}^{\,r}}{r!} R^r \Big\}:= C_{R,\nu}< \infty.
\end{align*} 
Hence the function $ \phi_{F_n}(z)$ is also bounded away from $0$ on the disk $D_R$. Also, relation \eqref{eq:lower-2} implies that the following power series (complex variable) converge to some analytic function as soon as $\abs{z} < R$;
\begin{equation}\label{eq:lower-4}
\sum_{r\ge1}\frac{\kappa_r(F_n)}{r!}(iz)^r, \quad \sum_{r\ge1}\frac{\kappa_r(G(\nu))}{r!}(iz)^r.
\end{equation}
Thus we come to the conclusion that the functions $\phi_{G(\nu)}(z)$ and $\phi_{F_n}(z)$ are analytic on the disk $D_R$. Moreover, there exists a constant $c >0$ such that $\abs{ \phi_{G(\nu)}(z) },  \abs{ \phi_{F_n}(z) } \ge c >0$ for every $z \in D_R$. This implies that on the disk $D_R$ there exist two analytic functions $g_n$ and $g_\nu$ such that
\[
\phi_{F_n}(z)=e^{g_n (z)}, \quad \phi_{G(\nu)}(z)=e^{g_\nu (z)},
\]
i.e. $g_n (z)= \log (\phi_{F_n}(z))$ and $g_\nu(z)=\log(\phi_{G(\nu)}(z))$, for $z \in D_R$. In fact, the functions $g_n$ and $g_\nu$ are given by the power series \eqref{eq:lower-4}. Since the derivative of the analytic branch of the complex logarithm is  $(\log z)' = \frac{1}{z}$ (see \cite[Corollary $2.21$]{conway}), one can infer that for some constant $C$, whose value may differ from line to line, and for every $z \in D_R$, we have 
\begin{align*}
\abs[\Big]{ \sum_{r \ge 2} \frac{\kappa_r(F_n) - \kappa_r (G(\nu))}{r!}(iz)^r } & = \abs[\Big]{  \log (\phi_{F_n}(z)) - \log(\phi_{G(\nu)}(z)) } \\
& \le_C \abs[\Big]{ \phi_{F_n}(z) - \phi_{G(\nu)}(z) } \le_C d_{2} (F_n, G(\nu)).
\end{align*}
 Now, using Cauchy's estimate for the coefficients of analytic functions, for any $r \ge 3$, we obtain that
 \[
 \abs[\Big]{ \kappa_r (F_n) - \kappa_r (G(\nu)) } \le r! R^r \sup_{\abs{z} \le R} \abs[\Big]{ \log \phi_{F_n} (z) - \log \phi_{G(\nu)}(z) }.
 \]
Therefore,
\[
\max \Big\{  \abs[\Big]{ \kappa_3 (F_n) - \kappa_3 (G(\nu)) }, \abs[\Big]{ \kappa_4 (F_n) - \kappa_4 (G(\nu)) } \Big\} \le_{C} d_{2} (F_n, G(\nu)) .
\]
\end{proof}

To demonstrate the power of Theorem \ref{thm:optimal}, we consider a second order U-statistic with degeneracy order $1$. The following example is taken from \cite[Section~3.1]{a-a-p-s}.
\begin{ex}
Let $\{h_i \}_{i \geq 1}$ be an orthonormal basis of $\HH$ and for $i \geq 1$ set $Z_i := I_1(h_i)$. For $a \neq 0$ consider
\[ U_n = \frac{2a}{n (n-1)} \sum_{1 \leq i < j \leq n} Z_i Z_j = I_2 \bigg( \frac{2a}{n (n-1)} \sum_{1 \leq i < j \leq n} h_i \widetilde{\otimes} h_j \bigg). \]
Then $nU_n \stackrel{\mathcal{D}}{\to} a (Z_1^2 -1)$ as $n \to \infty$. Since the target is only distributed according to a centered Gamma distribution if $a=1$, we will restrict ourselves to this case and write $G(1)$ for the target. Furthermore, in our setting, we need to fix the variance of our sequence to $2$. Hence we consider
\begin{align*}
W_n := \sqrt{\frac{n-1}{n}} n U_n & = I_2 \bigg(  \frac{2}{\sqrt{n (n-1)}} \sum_{1 \leq i < j \leq n} h_i \symtensor h_j \bigg)  \\
& = I_2 \bigg( \frac{1}{\sqrt{n (n-1)}} \sum_{\substack{1 \leq i,j \leq n \\ i \neq j}} h_i \tensor h_j \bigg) =: I_2(f_n)
\end{align*}
We consider the associated Hilbert-Schmidt operator $A_{f_n} g = f_n \cont{1} g$. Using the fact that $(h_i \tensor h_j) \cont{1} h_k = \sprod{h_i, h_k}_{\HH} \, h_j$ we can explicitly compute the non-zero eigenvalues $c_{1,n}, \ldots, c_{n,n}$ of $A_{f_n}$. They are
\[ c_{1,n} = \sqrt{\frac{n-1}{n}}, \text{ and } c_{2,n} = \ldots = c_{n,n} = \frac{-1}{\sqrt{n (n-1)}}.  \]
Since our target has $1$ degree of freedom, the assumptions of Theorem \ref{thm:optimal} are in order (see Proposition \ref{lem:finite-eigenvalue-1}(a)) and thus the optimality result \eqref{eq:optimal-finite-eigenvalue} holds for $W_n$. Also, with the eigenvalues given above and Lemma \ref{lem:Var(Gamma_r-2Gamma_r-1)}, one may verify manually that $\Var(\Gamma_3(W_n) - 2 \Gamma_2(W_n) ) \approx_C \Var^2(\Gamma_1(W_n) - 2 W_n) \approx_C \frac{1}{n^2}$. As a consequence
\[ 
d_2\big(W_n, G(1) \big) \approx_C \abs[\big]{\kappa_3(W_n) - \kappa_3(G(1))} \approx_C \abs[\big]{\kappa_4(W_n) - \kappa_4(G(1))} \approx_C \frac{1}{n}. \]
\end{ex}

\subsection{Trace Class Operators}\label{sec:trace-class}
\begin{lem}\label{lem:><trace}
Let $F= I_2(f)$ be a random element in the second Wiener chaos such that $A^4_f - A^3_f \ge 0$ (or $\le 0$) is a non-negative (or non-positive) operator, where $A_f$ is the associated Hilbert-Schmidt operator. Then
\[
\Var \left( \Gamma_3 (F) - 2 \Gamma_2 (F) \right) \le 2 \times 3!^2 \, \Big( \kappa_4(F) -6 \kappa_3(F) \Big)^2.
\]
\end{lem}
\begin{proof} Using relation \eqref{eq:3},  and the main result of \cite{Liu}, one can write
\begin{align*}
 \Var \left( \Gamma_3 (F) - 2 \Gamma_2 (F) \right) &= \frac{1}{7!} \kappa_8 (F) - \frac{4}{6!} \kappa_7(F) + \frac{4}{5!} \kappa_6(F) \\
 & \hspace{-3em} = 2^7 \, \Tr (A^8_f) - 2^8 \, \Tr (A^7_f) + 2^7 \, \Tr (A^6_f)  = 2^7 \, \Tr (A^8_f - 2 A^7_f + A^6_f) \\
 & \hspace{-3em} = 2^7 \, \Tr \left( (A^4_f - A^3_f )^2 \right) \le 2^7  \left( \Tr( A^4_f - A^3_f) \right)^2
 = 2 \times 3!^2 \, \Big( \kappa_4(F) -6 \kappa_3(F) \Big)^2.
\end{align*}
\end{proof}

Now, we can state the following \textbf{non asymptotic} version of the optimal rate of convergence towards the centered Gamma distribution $G(\nu)$.
\begin{prop}\label{prop:crazy}
Let $\nu >0$. Assume that $F=I_2 (f)$ is a random element in the second Wiener chaos such that $\E(F^2) =2 \nu$. Moreover, assume that $A^4_{f} - A^3_{f} \ge 0$  (or $\le 0$). Then there exist two constants $0 < C_1 < C_2$, such that
\begin{equation}\label{eq:optimal-finite-eigenvalue2}
C_1 \,  \mathbf{M} (F) 
\le d_{2} (F, G(\nu)) \le C_2 \,  \mathbf{M} (F),
\end{equation}
where, as before, $ \mathbf{M}(F) :=  \max \Big\{ \abs[\Big]{ \kappa_3 (F) - \kappa_3 (G(\nu)) }, \abs[\Big]{ \kappa_4 (F) - \kappa_4 (G(\nu)) }  \Big\}$.  
\end{prop}

\begin{proof}
For the upper bound, combine Theorem \ref{thm:MainMalliavinSteinBound} together with Lemma \ref{lem:r+1<r} and Lemma \ref{lem:><trace}. The lower bound is derived from Theorem \ref{thm:optimal}. 
\end{proof}

We close this section with two lemmas of independent interests. The first lemma gathers some non-asymptotic variance-estimates and will be used in the proof of Lemma \ref{lem:chain} in Section \ref{sec:gamma-charac}. The second lemma displays that differences of all higher cumulants can be controlled from above by the quantity $\textbf{M}(F)$. 
\begin{lem}\label{lem:V-etimate}
Let $F=I_2(f)$ be a general element in the second Wiener chaos. Then, for $r \ge 1$, the following estimates hold.
\begin{equation}\label{eq:Vestimate-1}
\Var^{\, 2} \left( \Gamma_{r+1}(F) - 2 \Gamma_r(F) \right) \le_C \Var \left( \Gamma_r(F) - 2 \Gamma_{r-1}(F) \right) \times \Var \left( \Gamma_{r+2}(F) - 2 \Gamma_{r+1}(F) \right),
\end{equation}
\begin{equation}\label{eq:Vestimate-2}
\Var^{\, 2r} \left( \Gamma_{2}(F) - 2 \Gamma_1(F) \right) \le_C \Var^{\, 2r-1} \left( \Gamma_1(F) - 2 F \right) \times \Var \left( \Gamma_{2r+1}(F) - 2 \Gamma_{2r}(F) \right),
\end{equation}
where the general constant $C$ is independent of $F$. In particular,

\begin{equation*}
\Var^{\, 2} \left( \Gamma_{2}(F) - 2 \Gamma_1(F) \right) \le_C \Var \left( \Gamma_1(F) - 2 F \right) \times \Var \left( \Gamma_{3}(F) - 2 \Gamma_{2}(F) \right).
\end{equation*}
Moreover,
\begin{align*}
\Var & \Big( \left(  \Gamma_{2r+1}(F) - 2 \Gamma_{2r}(F) \right) -2 \left(  \Gamma_{2r}(F) - 2 \Gamma_{2r-1}(F) \right)  \Big)  \le_C  \Var^{\, 2} \left(  \Gamma_{r}(F) -2 \Gamma_{r-1} (F) \right)\\
&  \qquad \le_C \Var \left( \Gamma_{r-1}(F) - 2 \Gamma_{r-2}(F) \right) \times \Var \left( \Gamma_{r+1}(F) - 2 \Gamma_{r}(F) \right).
\end{align*}

\end{lem}
\begin{proof}
This is a direct application of \cite[Corollary 1]{D-trace-1} with $P=(A^{r+1}_f - A^{r}_f)^2, C=A^2_f$, and the fact that, for $r \ge 0$, we have $$\Var\left( \Gamma_{r+1}(F) - 2 \Gamma_{r}(F) \right) = 2^{2r+3} \Tr \left( (A^{r+2}_f - A^{r+1}_f)^2 \right).$$ The estimate \eqref{eq:Vestimate-2} is also an application of \cite[Corollary 1]{D-trace-2} with $P=(A^2_f - A_f)^2$, and the convex function $f(x)=x^{2r}$.
\end{proof}

\begin{lem}\label{lem:cumulant-difference}
Let $\nu >0$, and $F= I_2 (f)$ in the second Wiener chaos so that $\E[F^2]=2\nu$. Then, for every $r\ge 1$, there exists a constant $C$ (depending only on $\nu$, and $r$) such that
                                                                                
\begin{equation}\label{eq:gamma-difference}
 \abs[\Big]{\E [\Gamma_{r+1} (F)] - 2 \E [\Gamma_r (F)] }  \le_C \mathbf{M} (F),
\end{equation}
and also, 
\begin{equation}\label{eq:cumulant-difference}
\abs[\Big]{  \kappa_{r} (F) - \kappa_{r} (G(\nu)) } \le_C  \mathbf{M} (F).
\end{equation}
\end{lem}
\begin{proof}
We proof estimate \eqref{eq:gamma-difference} by induction on $r$. The estimate \eqref{eq:cumulant-difference} can be derived in a similar way. Obviously \eqref{eq:gamma-difference} holds for $r=1,2$, so we assume that $r \ge 3$. Note that 
\begin{align*}
\abs[\Big]{ \E [\Gamma_{r+2} (F)] - & 2 \E [\Gamma_{r+1}(F)] }  =  \abs[\Big]{ \frac{\kappa_{r+3}(F)}{(r+2)!} - 2 \frac{\kappa_{r+2}(F)}{(r+1)!} } \\
 & \le_C \abs[\Big]{ \frac{\kappa_{r+3}(F)}{(r+2)!} - 4\frac{\kappa_{r+2}(F)}{(r+1)!} + 4\frac{\kappa_{r+1}(F)}{r!} } +  \abs[\Big]{ \frac{\kappa_{r+2}(F)}{(r+1)!} - 2 \frac{\kappa_{r+1}(F)}{r!} } \\
 &=_C   \abs[\Big]{ \frac{\kappa_{r+3}(F)}{(r+2)!} - 4\frac{\kappa_{r+2}(F)}{(r+1)!} + 4\frac{\kappa_{r+1}(F)}{r!} } +   \abs[\Big]{ \E [\Gamma_{r+1} (F)] - 2 \E [\Gamma_r (F)] }.
\end{align*}
The second summand on the right hand side can be handled with the induction hypothesis. For the first summand, we have two possibilities. If $r=2s+1$, for some $s \ge 1$, then 
$$ \frac{\kappa_{r+3}(F)}{(r+2)!} - 4\frac{\kappa_{r+2}(F)}{(r+1)!} + 4\frac{\kappa_{r+1}(F)}{r!} =  \Var \left( \Gamma_{s+1}(F) - 2 \Gamma_{s} (F) \right) \le_C  \Var \left( \Gamma_{1}(F) - 2 F \right),$$ and so we are done. Otherwise, $r=2s$ for some $s \ge 2$. Hence, using Cauchy--Schwarz inequality, we obtain that 

\begin{align*}
 \abs[\Big]{ \frac{\kappa_{r+3}(F)}{(r+2)!} & - 4\frac{\kappa_{r+2}(F)}{(r+1)!} + 4\frac{\kappa_{r+1}(F)}{r!} } = \abs[\Big]{ \frac{\kappa_{2s+3}(F)}{(2s+2)!} - 4\frac{\kappa_{2s+2}(F)}{(2s+1)!} + 4\frac{\kappa_{2s+1}(F)}{(2s)!} } \\
 &= \abs[\Big]{ \E \Big(  \left(    \CenteredGamma_{s+1}(F) -  2 \CenteredGamma_{s}(F)  \right)  \times  \left(  \CenteredGamma_{s}(F) -  2 \CenteredGamma_{s-1}(F)  \right)\Big) } \\
 &\le_C \sqrt{ \Var \left( \Gamma_{s+1}(F) - 2 \Gamma_{s} (F) \right) } \times \sqrt{ \Var \left( \Gamma_{s}(F) - 2 \Gamma_{s-1} (F) \right) }\\
 &\le_C  \Var \left( \Gamma_{1}(F) - 2 F \right).
\end{align*}
For the estimate \eqref{eq:cumulant-difference}, note that $\kappa_{r+1}(G(\nu)) = 2 r \kappa_{r}(G(\nu))$. Therefore,
 \begin{align*}
\abs[\Big]{ \kappa_{r+1}(F) - \kappa_{r+1}(G(\nu)) } & \le \abs[\Big]{ \kappa_{r+1}(F) - 2r \kappa_{r}(F) } + \abs[\Big]{ 2r \kappa_{r}(F) - \kappa_{r+1}(G(\nu)) } \\
& \le_C \abs[\Big]{  \E [\Gamma_{r} (F)] -  2 \E [\Gamma_{r-1}(F)] } + \abs[\Big]{ \kappa_{r}(F) - \kappa_{r}(G(\nu)) }.
\end{align*}
\end{proof}

\subsection{A Further Example: Optimal Rate in Total Variation Distance}\label{sec:concrete-example}

In this section we introduce a concrete example of a sequence within the second Wiener chaos. The corresponding Hilbert-Schmidt will only have two non-zero eigenvalues, both of which are converging to $1$. A crucial observation is that although the presented example lies out of the favorable regimes discussed in Section \ref{sec:2variance-estimate}, the optimal rate $ \mathbf{M}(F_n)$ insists to hold in total variation distance.

\begin{Prop}\label{Prop:concrete-example}
Consider the sequence $\{ F_n = c_{1,n} (N_1^2 -1) + c_{2,n} (N_2^2 -1) \}_{n\geq 1}$ in the second Wiener chaos where $c_{1,n} = \sqrt{1+ \frac{1}{n}}$ and $c_{2,n} = \sqrt{1 - \frac{1}{n}}$. Then
	\[ d_{TV} \big(F_n, G(2) \big)  \approx_C \max \Big\{ \abs[\big]{ \kappa_3(F_n) - \kappa_3(G(2)) } , \abs[\big]{ \kappa_4(F_n) - \kappa_4(G(2)) } \Big\} \approx_C \frac{1}{n^2}. \]
\end{Prop}

\begin{proof}
	First note that $$
\kappa_4(F_n) - \kappa_4(G(2)) = 48 \sum_{j=1}^{2} \left( c_{j,n}^4 - 1 \right) = 48\,  \frac{2}{n^2} \approx_C \frac{1}{n^2}.$$ Similarly $
\kappa_3(F_n) - \kappa_3(G(2)) = 8 \sum_{j=1}^{2} \left( c_{j,n}^3 - 1 \right) \approx_C \frac{1}{n^2}$. To shorten notation, we write $c_1$ and $c_2$ instead of $c_{1,n}$ and $c_{2,n}$. We start by computing $\phi_n$, the probability density function of $F_n$. The density of $F_n$ is given by
	\begin{align*}
	\phi_n(x) = {} & \frac{1}{2 \pi \sqrt{c_1 c_2}} \, \int\limits_{-c_1}^{x+c_2} \frac{e^{-\frac{1}{2} \left( \frac{t + c_1}{c_1} + \frac{x - t + c_2}{c_2}  \right) }}{\sqrt{t + c_1} \, \sqrt{x - t + c_2}} \, dt \, \ind{x > - c_1 - c_2}(x) \\
	{} = {} & \frac{1}{2 \pi \sqrt{c_1 c_2}} \, e^{-  1 - \frac{x}{2 c_2} } \, \int\limits_{-c_1}^{x+c_2} \frac{e^{\frac{t}{2} \left( \frac{1}{c_2} - \frac{1}{c_1} \right)}}{\sqrt{t + c_1} \, \sqrt{x - t + c_2}} \, dt \, \ind{x > - c_1 - c_2}(x).
	\end{align*}
	Substituting $t=(c_1 + c_2 + c) u - c_1$, we get
	\begin{align*}
	\phi_n(x) = {} & \frac{1}{2 \pi \sqrt{c_1 c_2}} \, e^{-  1 - \frac{x}{2 c_2} } \, \int\limits_{0}^{1} \frac{ e^{\frac{(c_1 - c_2) (c_1 + c_2 + x) u}{2 c_1 c_2} - \frac{c_1 - c_2}{2 c_2} } (x+ c_1 + c_2) }{\sqrt{u(x+c_1 +c_2)} \, \sqrt{(1-u) (x+ c_1 + c_2)}} \, du \, \ind{x > - c_1 - c_2}(x) \\
	{} = {}      & \frac{1}{2 \sqrt{c_1 c_2}} \, e^{ - \frac{x + c_1}{2 c_2} - \frac{1}{2} } \times \frac{1}{\pi} \int\limits_{0}^{1} e^{\frac{(c_1 - c_2) (c_1 + c_2 + x) }{2 c_1 c_2} \, u} \, \frac{1}{\sqrt{u}} \, \frac{1}{\sqrt{1-u}} \, du \, \ind{x > - c_1 - c_2}(x)                         \\
	{} = {}      & \frac{1}{2 \sqrt{c_1 c_2}} \, e^{ - \frac{x + c_1}{2 c_2} - \frac{1}{2} } \times \confhyper\Big(\frac{1}{2},1, \frac{c_1 - c_2}{2 c_1 c_2} (c_1+c_2+x) \Big) \times \ind{x > - c_1 - c_2}(x).
	\end{align*}
	Here, $\confhyper$ is the confluent hypergeometric function, which can be represented as
	\[ \confhyper(a,b,z) = {\frac  {\Gamma (b)}{\Gamma (a)\Gamma (b-a)}}\int _{0}^{1}e^{{zu}}u^{{a-1}}(1-u)^{{b-a-1}}\,du \]
	for $\Re(b) > \Re(a) > 0$. Note that if $a= 1/2$ and $b=1$, we get
	\[ \confhyper \Big(\frac{1}{2},1,z \Big) = {\frac{\Gamma (1)}{\Gamma (\frac{1}{2})^2}}\int _{0}^{1} e^{zu} \frac{1}{\sqrt{u}} \, \frac{1}{\sqrt{1-u}} \,du =  {\frac{1}{\pi} }\int _{0}^{1} e^{zu} \frac{1}{\sqrt{u}} \, \frac{1}{\sqrt{1-u}} \,du. \]
	Also note that the roles of $c_1$ and $c_2$ are completely interchangeable. It is just a matter of how we write down the convolution. Thus we can also write
	\[
	\phi_n(x) = \frac{1}{2 \sqrt{c_1 c_2}} \, e^{ - \frac{x + c_1 + c_2}{2 c_1} } \times \confhyper\Big(\frac{1}{2},1, - \frac{c_1 - c_2}{2 c_1 c_2} (c_1+c_2+x) \Big) \times \ind{x > - c_1 - c_2}(x).
	\]
	Also recall that the density of the target $G(2)$ is given by
	\[ \psi(x) = \frac{1}{2} e^{-\frac{x}{2} -1 } \ind{x > -2}(x). \]
	The next step is to explicitly write down the total variation distance in terms of the density functions:
	\begin{align*}
	d_{TV}(F_n, G(2)) & = \frac{1}{2} \int_{-\infty}^{\infty} \abs{ \phi_n(x) - \psi(x)  } \, dx \\
	& = \frac{1}{2} \int_{- 2}^{-c_1 - c_2} \psi(x) \, dx + \frac{1}{2} \int_{- c_1 - c_2}^{\infty} \phi_n(x) - \psi(x) \,  dx \\
	& = \frac{1}{2} \left( 1- e^{\frac{c_1 + c_2}{2} -1} \right) + \frac{1}{2} \int_{- c_1 - c_2}^{\infty} \phi_n(x) - \psi(x) \,  dx \\
	& =: \frac{1}{2} \big( \alpha_1(n) + \alpha_2(n) \big).
	\end{align*}
	One can readily check that $\alpha_1(n) \approx_C \frac{1}{n^2}$. To examine the asymptotic behaviour of $\alpha_2(n)$, we write
	\begin{align*}
	& \phi_n(x) - \psi(x) \\
	& \qquad = \frac{1}{2} \, e^{-\frac{x}{2}} \left[ \frac{1}{\sqrt{c_1 c_2}} \, e^{- \frac{c_1+c_2}{2 c_1}} \, e^{x\left( \frac{1}{2} - \frac{1}{2 c_1} \right)} \, \confhyper\Big(\frac{1}{2},1, - \frac{c_1 - c_2}{2 c_1 c_2} (c_1+c_2+x) \Big) - e^{-1} \right], \\
	\end{align*}
	and find a series expansion for the term inside the square brackets.
	Expanding $\confhyper$ as a series (see e.g. \cite[p.~504]{Abramowitz.Stegun}), we get
	\[
	\confhyper\Big(\frac{1}{2},1, - \frac{c_1 - c_2}{2 c_1 c_2} (c_1+c_2+x) \Big) = \frac{1}{\sqrt{\pi}} \sum_{k=0}^{\infty} \frac{\Gamma( \frac{1}{2} +k)}{\Gamma(1+k)} \frac{\left[ - \frac{c_1 - c_2}{2 c_1 c_2} (c_1+c_2+x)  \right]^k}{k!}.
	\]
	On the other hand, we can expand the exponential around $-c_1-c_2$ as
	\[
	e^{x\left( \frac{1}{2} - \frac{1}{2 c_1} \right)} = e^{( -c_1 - c_2 ) \left( \frac{1}{2} - \frac{1}{2 c_1} \right)} \sum_{k=0}^{\infty} \frac{ \left( \frac{1}{2} - \frac{1}{2 c_1} \right)^k ( x + c_1 +c_2 )^k }{k!}.
	\]
	Thus, we obtain the following series expansion
	\begin{align*}
	& e^{x\left( \frac{1}{2} - \frac{1}{2 c_1} \right)} \times \confhyper\Big(\frac{1}{2},1, - \frac{c_1 - c_2}{2 c_1 c_2} (c_1+c_2+x) \Big) \\
	{}={} & \frac{e^{( -c_1 - c_2 ) \left( \frac{1}{2} - \frac{1}{2 c_1} \right)} }{\sqrt{\pi}}  \sum_{k=0}^{\infty} \sum_{\ell=0}^{k} \frac{\Gamma( \frac{1}{2} + \ell)}{\Gamma(1+\ell)} \frac{\left[ - \frac{c_1 - c_2}{2 c_1 c_2} (c_1+c_2+x)  \right]^{\ell}}{\ell!} \frac{ \left[ \left( \frac{1}{2} - \frac{1}{2 c_1} \right) ( x + c_1 +c_2 ) \right]^{k-\ell} }{(k-\ell)!} \\
	{}={} & \frac{e^{( -c_1 - c_2 ) \left( \frac{1}{2} - \frac{1}{2 c_1} \right)} }{\sqrt{\pi}}  \sum_{k=0}^{\infty} \sum_{\ell=0}^{k} \frac{\Gamma( \frac{1}{2} + \ell)}{\Gamma(1+\ell)} \frac{ \left( - \frac{c_1 - c_2}{c_1 c_2} \right)^{\ell} \left( 1-\frac{1}{c_1} \right)^{k-\ell}}{2^k\, \ell! \, (k-\ell)!} \, (x+c_1 + c_2)^k \\
	{}=: &  \sum_{k=0}^{\infty} A_k(c_1,c_2) (x+c_1+c_2)^k .
	\end{align*}
	Now
	\begin{align*}
	2 \, \alpha_2(n) & = \int_{-c_1 - c_2 }^{\infty} e^{-\frac{x}{2}}  \left[ \frac{1}{\sqrt{c_1 c_2}} \, e^{- \frac{c_1+c_2}{2 c_1}} \, \left( \sum_{k=0}^{\infty} A_k(c_1,c_2) (x+c_1+c_2)^k \right) - e^{-1} \right] \, dx \\
	& = \left( \frac{e^{- \frac{1}{2} (c_1+c_2)  } }{\sqrt{c_1 c_2}} -e^{-1} \right) \int_{-c_1 - c_2 }^{\infty} e^{-\frac{x}{2}}  \, dx \\
	& \quad + \sum_{k=1}^{\infty}  \frac{1}{\sqrt{c_1 c_2}} \, e^{- \frac{c_1+c_2}{2 c_1}} \, A_k(c_1,c_2) \int_{-c_1 -c_2}^{\infty} e^{-\frac{x}{2}} (x+c_1+c_2)^k \, dx.
	\end{align*}
	Using the fact that $\int_{-c_1 - c_2}^{\infty} \exp(-\frac{x}{2}) (x+c_1 + c_2)^k = k! \, 2^{k+1} \exp( \frac{c_1+c_2}{2})$ for all $k \in \N_0$,
	and setting
	\begin{align*}
	B_1(c_1,c_2) & := 2 \, e^{\frac{c_1 + c_2}{2} } \qquad \text{and} \\
	B_2(c_1,c_2,k) & :=  \frac{1}{\sqrt{c_1 c_2}} \, e^{- \frac{c_1+c_2}{2 c_1}} \times k! \, 2^{k+1} e^{ \frac{c_1+c_2}{2}} \times  \frac{e^{( -c_1 - c_2 ) \left( \frac{1}{2} - \frac{1}{2 c_1} \right)} }{\sqrt{\pi}\, 2^k}  = \frac{2 \, k!}{\sqrt{c_1 c_2} \sqrt{\pi}},
	\end{align*}
	we get
	\begin{align*}
	2 \, \alpha_2(n) {}={} & B_1(c_1,c_2)  \left( \frac{e^{- \frac{1}{2} (c_1+c_2) } }{\sqrt{c_1 c_2}} -e^{-1} \right) \\
	&  + \sum_{k=1}^{\infty} B_2(c_1,c_2,k) \left( \, \sum_{\ell=0}^{k} \frac{\Gamma( \frac{1}{2} + \ell)}{\Gamma(1+\ell)} \frac{ \left( - \frac{c_1 - c_2}{c_1 c_2} \right)^{\ell} \left( 1-\frac{1}{c_1} \right)^{k-\ell}}{\ell! \, (k-\ell)!} \right),
	\end{align*}
	where $B_1(c_1,c_2)$ and $B_2(c_1,c_2,k)$ converge (for fixed $k$) to a positive constant as $n \to \infty$, and thus do not contribute to the rate of convergence.	One can easily check that $\frac{e^{- \frac{1}{2} (c_1+c_2) } }{\sqrt{c_1 c_2}} -e^{-1} \approx_C \frac{1}{n^2} $ as $n \to \infty$. All the other terms are of the form ``something that converges to a constant'' $\times$ ``a polynomial in $c_1$ and $c_2$''. However, the terms for $k=1$ and $k=2$ also have the same rate of convergence, whereas the terms for $k\geq 3$ converge to zero at a faster rate. More precisely, we have
	\begin{align*}
	2 \, \alpha_2(n) & = 2 \left( \frac{1}{\sqrt{c_1 c_2}} - e^{\frac{c_1+c_2}{2} -1 }\right) + \frac{1}{\sqrt{c_1 c_2}c_1 c_2} \left( 2c_1c_2 - c_1 -c_2\right) \\
	& \qquad + \frac{1}{4 \sqrt{c_1 c_2}c_1^2 c_2^2} \left(8 c_1^2 c_2^2 - 8 c_1^2 c_2 + 3 c_1^2 - 8 c_1 c_2^2 + 2 c_1c_2 + 3 c_2^2\right) \\
	& \qquad + \sum_{k=3}^{\infty} B_2(c_1,c_2,k) \left( \, \sum_{\ell=0}^{k} \frac{\Gamma( \frac{1}{2} + \ell)}{\Gamma(1+\ell)} \frac{ \left( - \frac{c_1 - c_2}{c_1 c_2} \right)^{\ell} \left( 1-\frac{1}{c_1} \right)^{k-\ell}}{\ell! \, (k-\ell)!} \right) \\
	& =: C(c_1,c_2) + \sum_{k=3}^{\infty} B_2(c_1,c_2,k) \left( \, \sum_{\ell=0}^{k} \frac{\Gamma( \frac{1}{2} + \ell)}{\Gamma(1+\ell)} \frac{ \left( - \frac{c_1 - c_2}{c_1 c_2} \right)^{\ell} \left( 1-\frac{1}{c_1} \right)^{k-\ell}}{\ell! \, (k-\ell)!} \right).
	\end{align*}
	After some computations, we see that, as $n \to \infty$,	$ \frac{C(c_1,c_2)}{1/n^2} \to 1$, whereas the remaining terms converge faster.
	
\end{proof}

\section{Gamma Characterisation Within the Second Wiener Chaos}\label{sec:gamma-charac}


Let $\nu >0$ and $G(\nu)$ be a centered Gamma distributed random variable. Assume that $F$ is a random element in the second Wiener chaos such that $\E[F^2]= 2 \nu$. The proof of Proposition \ref{Prop:H3-metric} reveals that 
\begin{equation}\label{eq:5}
\begin{split}
d_{2} \big( F,G(\nu) \big) \le_C \bigg\{  \abs[\big]{ \kappa_3(F) - \kappa_3 (G(\nu)) } +\sqrt{ \Var \left( \Gamma_2 (F) - 2 \Gamma_1 (F) \right)} \bigg\}.
\end{split}
\end{equation}
 From observation \eqref{eq:5}, it appears that the sole condition $\Var \left( \Gamma_2 (F) - 2 \Gamma_1 (F) \right)=0$ may not be enough to conclude that the random variable $F$ is distributed like $G (\nu)$, and in addition, one needs to match the third cumulants $\kappa_3 (F)= \kappa_3 (G(\nu))$.   A simple example \textbf{outside} the second Wiener chaos is $F \sim \mathscr{N}(0,1)$. Then, obviously, $\Var \left( \Gamma_2 (F) - 2 \Gamma_1 (F) \right)=0$, but $F$ is not Gamma distributed. Note that $\kappa_3(F)=0$, whereas $\kappa_3 (G(\nu)) \neq 0$. The next lemma clarifies that this is not the case.

\begin{lem}\label{lem:chain}
Let $r \ge 0$ be an integer and suppose that $F= I_2 (f)$ belongs to the second Wiener chaos, such that $\E[F^2]=2\nu$, where $\nu > 0$. Then 
\[
F \stackrel{\text{law}}{=} G(\nu) \quad \text{if and only if} \quad \Delta_{r}(F) := \Var\left( \Gamma_{r+1}(F) - 2 \Gamma_r(F) \right) =0, \, \textbf{ for some } \, r \ge 0.
\] 
Moreover
\begin{equation}\label{eq:chain}
\cdots \le_C \Delta_{r+1}(F) \le_C \Delta_{r}(F) \le_C \Delta_{r-1} (F) \le_C \cdots \le_C \Delta_0 (F) = \Var \left( \Gamma_{1}(F) - 2 F \right). 
\end{equation}
\end{lem}
\begin{proof}
The chain of estimates in \eqref{eq:chain} follows from Lemma \ref{lem:r+1<r}. Also, it is well known that if $F \sim CenteredGamma(\nu)$, then $\Delta_0 (F) =0$, and therefore $\Delta_r (F) =0$ for all $r \ge 1$. For the other direction let 
\[
F = \sum_{i \ge 1} c_i (N^2_i -1), \quad \text{and} \quad \E[F^2]= 2 \sum_{i \ge 1} c^2_i =2 \nu.
\]
Assume that
\begin{align*}
0= \Delta_1 (F) = \Var & \left( \Gamma_2 (F) - 2 \Gamma_1 (F) \right)= \frac{\kappa_6(F)}{5!} - \frac{4}{4!} \kappa_5(F) + \frac{4}{3!} \kappa_4 (F)\\
&= 2^5 \sum c^6_i - 4 \times 2^4 \sum c^5_i + 4 \times 2^3 \sum c^4_i \\
&= 2 \times \sum \left( 2^5  c^6_i - 4 \times 2^4 c^5_i + 4 \times 2^3 c^4_i \right)\\
&= \sum \left( 2^3 c^3_i - 2^3 c^2_i \right)^2.
\end{align*}
Hence, we either have $c_i =0$ for all $i \ge 1$, which is impossible, or $c_i =1$ for all $i \ge 1$. Now together with the condition $ 2 \sum_{i \ge 1} c^2_i =2 \nu$, we can deduce that there are only finitely many non-zero coefficients $c_i$, and moreover that $\nu$ is an integer. Hence
\[
F = \sum_{i=1}^{\nu} (N^2_i -1).
\]
The general case $r \ge 2$ follows from Lemma \ref{lem:V-etimate}.
\end{proof}

For $r,\lambda  \in (0,\infty)$, in what follows, we denote by $\Gamma(r,\lambda)$, the  \textit{Gamma} distribution with \textit{shape parameter} $r$, and \textit{rate} $\lambda$, which has the following probability density function 
\begin{equation*}
p_{r,\lambda} (x) = \begin{cases} \frac{\lambda^r}{\Gamma(r)} x^{r-1} e^{-\lambda x}, & \text{ if } x \ge 0,\\
0, & \text{otherwise},
\end{cases}
\end{equation*}
where $\Gamma(r)$ stands for \textit{Gamma} function. The side goal of the next lemma is to provide other sufficient conditions for the validity of variance-estimate \eqref{variance-estimate-2}.  
\begin{lem}\label{lem:3<1??}
Let $\nu > 0$. Assume that $F= I_2 (f)$ is in the second Wiener chaos such that $\E[F^2]=2\nu$, and moreover $\text{Cov}\left( \Gamma_3 (F) - 2 \Gamma_2 (F), \Gamma_2 (F) - 2 \Gamma_1 (F)\right) \ge 0$ (see Remark \ref{rem:3<1} item (iv)). For $\beta \in \R$, define the \textit{biquadratic} function 
\[
\Phi_F (\beta):= \Var \Big( \left( \Gamma_3 (F) - 2 \Gamma_2 (F) \right) -2 \beta^2 \left( \Gamma_2(F) -2 \Gamma_1(F) \right) \Big).
\] 
Then, function $\Phi_F$ attains its local maximum at $\beta_{\text{max}}=0$, and the global minimum at the points
\[
\beta^{\pm}_{\text{min}} = \pm \sqrt{\frac{\E \big[ \left( \CenteredGamma_3(F) - 2 \CenteredGamma_2(F) \right)  \left( \CenteredGamma_2(F) - 2 \CenteredGamma_1 (F) \right) \big] }{ 2 \Var \left( \Gamma_2 (F) - 2 \Gamma_1 (F) \right)}}.
\] 
Also, the following statements are equivalent.
\begin{itemize}
\item[(a)]  $F \stackrel{\mathcal{D}}{=} G(\nu)$.
\item[(b)]  $\Phi_F(\beta)=0$, for some $\beta \neq \beta^{\pm}_{\text{min}}$.
\end{itemize}
In particular, $\Phi_F(1)=0$ implies that $F \stackrel{\mathcal{D}}{=} G(\nu)$, and
\begin{equation}\label{eq:beta=1}
0 \le \Phi_F(1) =  \frac{\kappa_8 (F)}{7!} - 8 \frac{\kappa_7 (F)}{6!} + 24 \frac{\kappa_6 (F)}{5!} - 32 \frac{\kappa_5 (F)}{4!} +16 \frac{\kappa_4 (F)}{3!} \le 2   \Var^{\, 2} \left( \Gamma_1 (F) - 2 F \right).
\end{equation}
Let 
\begin{equation}\label{eq:beta0}
\beta_0 := \pm  \sqrt{\frac{\E \big[ \left( \CenteredGamma_3(F) - 2 \CenteredGamma_2(F) \right)  \left( \CenteredGamma_2(F) - 2 \CenteredGamma_1 (F) \right) \big] }{  \Var \left( \Gamma_2 (F) - 2 \Gamma_1 (F) \right)}}.
\end{equation}
\begin{itemize}
\item[$\mathbf{(I)}$] In the case $\beta_0 \le 1$,
\begin{equation*}
\Var \left( \Gamma_3 (F) - 2 \Gamma_2 (F) \right) \le 2 \Var^{ \, 2} \left( \Gamma_1 (F) - 2 F \right).
\end{equation*}
\item[$\mathbf{(II)}$] In the case $\beta_0 \ge 1$,
\begin{align*}
 \Phi_F(\beta^-_{\text{min}}) &=  \Phi_F(\beta^+_{\text{min}})  = \Var \left( \Gamma_3 (F)  - 2 \Gamma_2 (F) \right)  \\
 & \hskip3cm - \frac{ \Big(\E [ \left( \CenteredGamma_3(F) - 2 \CenteredGamma_2(F) \right)  \left( \CenteredGamma_2(F) - 2 \CenteredGamma_1 (F) \right)] \Big)^2}{\Var  \left( \Gamma_2(F) - 2 \Gamma_1 (F) \right)}\\
 & \le \Var \Big( \left( \Gamma_3 (F) - 2 \Gamma_2 (F) \right) -2  \left( \Gamma_2(F) -2 \Gamma_1(F) \right) \Big) \le \Var \left( \Gamma_3 (F) - 2 \Gamma_2 (F) \right).
\end{align*}
\end{itemize}
Lastly, $\Phi_F(\beta^{\pm}_{\text{min}})=0$ implies that $F \sim \left[ 2 \Gamma( \frac{\ell_1}{2}, \frac{2}{k}) -\frac{\ell_1}{2k} \right] \star \left[ 2 \Gamma(\frac{\ell_2}{2},1) - \ell_2 \right]$, where $k=\beta^2_0$, $\beta_0$ is given by \eqref{eq:beta0}, and the operation $\star$ stands for convolution. $(\ell_1,\ell_2) \in \N^2_0$ are such that $\ell_1 \frac{k^2}{4} + \ell_2 =\nu$, i.e. 
\begin{equation}\label{eq:mixed-gamma}
F \stackrel{\text{law}}{=} \sum_{i=1}^{\ell_1} \frac{k}{2} (N^2_i -1 ) + \sum_{j=1}^{\ell_2} (N^2_j -1),
\end{equation}
where $\{N_i, N_j \,  : \, 1 \le i \le \ell_1 , 1 \le j \le \ell_2\}$ are i.i.d. $\mathscr{N}(0,1)$ variables with the convention that, when $\ell_1=1$ or $\ell_2 =0$, the corresponding sum is understood as $0$.
\end{lem}
\begin{proof}
 It is straightforward to deduce that the function $\Phi_F$ attains its local maximum at $\beta_{\text{max}}=0$, and the global minimum at points $\beta^{\pm}_{\text{min}}$. Also 
 \begin{equation}\label{eq:beta-min}
\Phi _F( \beta^{\pm}_{\text{min}})
 =\Var \left( \Gamma_3 (F) - 2 \Gamma_2 (F) \right) - \frac{ \Big(\E [ \left( \CenteredGamma_3(F) - 2 \CenteredGamma_2(F) \right)  \left( \CenteredGamma_2(F) - 2 \CenteredGamma_1 (F) \right)] \Big)^2}{\Var  \left( \Gamma_2(F) - 2 \Gamma_1 (F) \right)}.
 \end{equation}
$(a) \Rightarrow (b)$: It is obvious that when $F \stackrel{\mathcal{D}}{=} G(\nu)$, then $\Phi_F (0)=0$. In fact $\Phi_F (\beta)=0$ for all $\beta \in \R$, because
\[
0 \le \Phi_F(\beta) \le 2  \Var \left( \Gamma_3 (F) - 2 \Gamma_2 (F) \right) + 8 \beta^4 \Var \left( \Gamma_2 (F) - 2 \Gamma_1 (F) \right).
\] 
$(b) \Rightarrow (a)$: Assume that $\Phi_F ( \beta' ) =0$ for some $\beta' \neq  \beta^{\pm}_{\text{min}}$. Since $\Phi_F \ge 0$, this implies that $\Phi _F( \beta^{\pm}_{\text{min}})=0$. Hence, relation \eqref{eq:beta-min} yields that the equality case happens in the Cauchy--Swartz inequality.  Therefore, for some constant $k$ (in fact $k = 2 (\beta^{\pm}_{\text{min}})^2=\beta^2_0$), we have 
\begin{equation}\label{eq:=a.s.}
\Gamma_3 - 2 \, \Gamma_2 \stackrel{\text{a.s.}}{=} k \left( \Gamma_2 - 2 \Gamma_1 \right).
\end{equation}
Hence, 
\[
\Phi_F( \beta)= (k-2\beta^2)^2 \Var \left( \Gamma_2 (F) - 2 \Gamma_1 (F) \right) = \left( 2 (\beta^{\pm}_{\text{min}})^2 - 2 \beta^2 \right)^2  \Var  \left( \Gamma_2(F) - 2 \Gamma_1 (F) \right).
\]
Now, assumption $\Phi_F ( \beta' ) =0$ tells us that $\Var  \left( \Gamma_2(F) - 2 \Gamma_1 (F) \right)=0$, and so $F$ is distributed like $G(\nu)$. To continue the rest of the proof, let $F = \sum_i c_i (N^2_i -1)$ for some sequence of real numbers $(c_i)_{i \ge 1}$ such that $\sum_i c^2_i < \infty$. Then we know that 
\[
2  \Var\left( \Gamma_1 (F) - 2 F \right)  = 2 \frac{\kappa_4(F)}{3!} - 4 \frac{\kappa_3 (F)}{2!} + 4 \kappa_2(F)
= \sum_i \left( 2^2 c^2_i - 2^2 c_i \right)^2.
\]
Also
\[
2  \Var \left( \Gamma_3 (F) - 2 \Gamma_2(F) \right)  =  2\frac{\kappa_8(F)}{7!} - 4 \frac{\kappa_7 (F)}{6!} + 4 \frac{\kappa_6(F)}{5!}= \sum_i \left( 2^4 c^4_i - 2^4 c^3_i \right)^2.
\]
Hence,
\begin{align*}
\Big( 2  \Var & \left( \Gamma_1 (F) - 2 F \right) \Big)^2 = \Big(  \sum_i \left( 2^2 c^2_i - 2^2 c_i \right)^2\Big)^2\\
& = \sum_i \left( 2^2 c^2_i - 2^2 c_i \right)^4 +  \sum_{i \neq j} \left( 2^2 c^2_i - 2^2 c_i \right)^2  \left( 2^2 c^2_j - 2^2 c_j \right)^2\\
& = : A + B.
\end{align*}
Note that $B \ge 0$. Now
\begin{align*}
A &= \sum_i \Big( 2^4 c^4_i + 2^4 c^2_i - 2^5 c^3_i \Big)^2 = \sum_i \Big[   \left(    2^4 c^4_i  - 2^4 c^3_i  \right) - \left(   2^4 c^3_i -2^4 c^2_i  \right) \Big]^2\\
&=\sum_i  \left(    2^4 c^4_i  - 2^4 c^3_i  \right)^2 + \sum_i  \left(   2^4 c^3_i -2^4 c^2_i  \right)^2 - 2 \sum_i \left(    2^4 c^4_i  - 2^4 c^3_i  \right)  \left(   2^4 c^3_i -2^4 c^2_i  \right)\\
& = : A_1 + A_2 -2 A_3.
\end{align*}
Note that $A_1 = 2 \Var \left( \Gamma_3(F) -2 \Gamma_2(F) \right)$. Moreover, 
\[
2 \Var \left( \Gamma_2(F) -2 \Gamma_1(F) \right) =  2 \, \Big\{  \frac{\kappa_6(F)}{5!} - 4 \frac{\kappa_5 (F)}{4!} + 4 \frac{\kappa_4(F)}{3!} \Big\}
= \sum_i \left( 2^3 c^3_i - 2^3 c^2_i \right)^2.
\]
Therefore $A_2 = 8  \Var \left( \Gamma_2(F) -2 \Gamma_1(F) \right)$. Also,
\begin{align*}
A_3 & =  \sum_i \left(    2^4 c^4_i  - 2^4 c^3_i  \right)  \left(   2^4 c^3_i -2^4 c^2_i  \right) 
= 4 \Big\{ \frac{\kappa_7(F)}{6!} - 4 \, \frac{\kappa_6 (F)}{5!} + 4 \frac{\kappa_5(F)}{4!} \Big\}\\
&= 4 \, \E \Big[ \left( \CenteredGamma_3(F) -2 \CenteredGamma_2 (F) \right) \left( \CenteredGamma_2(F) -2 \CenteredGamma_1 (F) \right) \Big].
\end{align*}
Finally, we arrive in
\begin{align}
2   \Var^{\, 2}  \left( \Gamma_1 (F) - 2 F \right)  & =    \Var \left( \Gamma_3(F) -2 \Gamma_2(F) \right)  -4  \E [ \left( \CenteredGamma_3(F) -2 \CenteredGamma_2 (F) \right) \left( \CenteredGamma_2(F) -2 \CenteredGamma_1 (F) \right) ] \notag \\
& \quad + 4  \Var \left( \Gamma_2(F) -2 \Gamma_1(F) \right) +\frac{B}{2} \notag \\
&=  \Var \Big( \left( \Gamma_3 (F) - 2 \Gamma_2 (F) \right) -2  \left( \Gamma_2(F) -2 \Gamma_1(F) \right) \Big) +\frac{B}{2} \notag \\
& = \Phi_F (1) + \frac{B}{2}. \label{eq:funny}
\end{align}
As a result, $\Phi_F (1) \le 2  \Var^{\, 2}  \left( \Gamma_1 (F) - 2 F \right)$. Also, $\Phi_F (1)= \frac{A}{2}= 2^7 \sum_i (c^2_i - c_i)^4$. Hence, $\Phi_1 (F)=0$ implies that either $c_i =0$ for all $i \ge 1$, which is impossible, or $c_i =1$ for all $i \ge 1$. Now, together with the condition $ 2 \sum_{i \ge 1} c^2_i =2 \nu$, we can deduce that there are only finitely many non-zero coefficients $c_i$, and moreover that $\nu$ is an integer. Hence $F \stackrel{\mathcal{D}}{=} G (\nu)$ is centered Gamma distributed. \\
Case $\mathbf{(I)}$: Note that $\Phi_F (0) = \Phi_F (\beta_0) = \Var\left( \Gamma_3 (F) - 2 \Gamma_2 (F) \right)$, and also function $\Phi_F$ is increasing on $(-\infty, \beta^{-}_{\text{min}}) \cup (\beta^+_{\text{min}},+\infty)$. Since $\beta^+_{\text{min}} \le \beta_0$, this implies that, if $\beta_0 \le 1$,
\begin{align*}
\Var & \left( \Gamma_3 (F) - 2 \Gamma_2 (F) \right)= \Phi_F (0)= \Phi_F (\beta_0) \\
& \le \Phi_F (1) =  \Var \Big( \left( \Gamma_3 (F) - 2 \Gamma_2 (F) \right) -2  \left( \Gamma_2(F) -2 \Gamma_1(F) \right) \Big) \\
& \le 2  \Var^{\, 2}  \left( \Gamma_1 (F) - 2 F \right).
\end{align*}
Case $\mathbf{ (II)}$ can be discussed in a similar matter. Finally, if $\Phi_F (\beta^{\pm}_{\text{min}})=0$, then relation \eqref{eq:=a.s.} implies that $\Gamma_{r+1}(F) - 2 \Gamma_r (F) \stackrel{\text{a.s.}}{=} k^{r-1} \left( \Gamma_2 (F) - 2 \Gamma_1 (F) \right)$ for all $r \ge 2$. Note that we assume that $k \neq 0$, otherwise, one can immediately deduce that $F \stackrel{\mathcal{D}}{=} G(\nu)$. Therefore, 
\[
\sum_i \left( \frac{2c_i}{k} \right)^{r-1} (c^3_i - c^2_i) = \sum_i (c^3_i - c^2_i), \quad \forall \, r \ge 2.
\]
Hence, in general, the non-trivial possibility is that $c_i = k/2$ for some index $ 1 \le i \le \ell_1$, and $c_j=1$ for some other index $ 1\le j \le \ell_2$. Now taking into account the second moment assumption, we obtain that $\ell_1 \frac{k^2}{4} + \ell_2 = \nu$, i.e. 
\begin{equation}\label{eq:last-claim}
F= \sum_{i=1}^{\ell_1} k/2 \left(N^2_i - 1 \right)  + \sum_{j=1}^{\ell_2} (N^2_j -1).
\end{equation}
\end{proof}

\begin{rem}\label{rem:3<1}
\begin{itemize}
\item[\textit{(i)}] Plainly, when $k=2$, the random variable $F$ appearing in \eqref{eq:last-claim} is in fact distributed like $G(\nu)$. This is consistent with the fact that when $k=2$, relation \eqref{eq:=a.s.} turns into $\Gamma_3 (F) - 4 \Gamma_2(F) + 4 \Gamma_1 (F) = 0$. The latter means that $\Phi_F (1) =0$, and therefore $F \stackrel{\mathcal{D}}{=} G(\nu)$. Also note that $k^2 \le 4 \nu$, and so $\ell_1 \le 16$ and $\ell_2 \le \nu$.
\item[\textit{(ii)}]  By setting $ u  =  2 f \cont{1} f - 2 f$ and $v = 2^2 f \contIterated{1}{3} f - 2^2 f \cont{1} f$, one can verify that
 \begin{align*}
   u \cont{1} u &= \left( 2 f \cont{1} f - 2 f \right) \cont{1} \left( 2 f \cont{1} f - 2 f \right) \\
 &= 2^2 f \contIterated{1}{4} f - 2^3 f \contIterated{1}{3} f + 2^2 f \cont{1} f\\
 &= \Big\{ 2^2 f \contIterated{1}{4} f - 2^2 f \contIterated{1}{3} f \Big\} - \Big\{ 2^2 f  \contIterated{1}{3} f -  2^2 f \cont{1} f \Big\}\\
 &= \Big\{ 2^2 f \contIterated{1}{4} f - 2^2 f \contIterated{1}{3} f \Big\} -  v.
 \end{align*}
 Therefore,
 \begin{align*}
  \Phi_F (1) &=  \Var \Big(  \left( \Gamma_3(F) - 2 \Gamma_2 (F) \right) -2   \left( \Gamma_2(F) - 2 \Gamma_1 (F) \right) \Big)\\
 &= 8 \norm{ u \cont{1} u }^2 \\
 & \le  8 \norm{u}^4 = 2   \Var^{\, 2} \left( \Gamma_1 (F) - 2 F \right)
 \end{align*}
Hence, relying on relation \eqref{eq:funny}, the quantity $B/2$ is the exact amount one loses when applying the classical estimate $\norm{ f \cont{1} g } \le \norm{f} \, \norm{g}$.
%
\item[\textit{(iii)}] Let $A_f$ stand for the associated Hilbert-Schmidt operator. Condition $\beta_0 \le 1$ in item $\mathbf{(I)}$ of Lemma \ref{lem:3<1??} is equivalent to the following trace inequality
\[
\Tr\left( A_f (A^3_f - A^2_f)^2 \right) \le \frac{1}{2} \Tr ( A^3_f - A^2_f )^2.
\]
\item[(iv)] We also point out that the condition $\text{Cov}\left( \Gamma_3 (F) - 2 \Gamma_2 (F), \Gamma_2 (F) - 2 \Gamma_1 (F) \right) < 0$ yields that 
$$ \Var\left( \Gamma_3 (F) - 2 \Gamma_2 (F) \right) \le \Phi_F(1) \le 2 \Var^{\, 2} \left(\Gamma_1 (F) - 2 F \right).$$

\item[\textit{(v)}] Let $( F_n)_{n \ge 1}= \big( \sum_i c_{i,n} (N^2_i -1) \big)_{n\ge 1}$ be a sequence in the second Wiener chaos such that $\E(F^2_n)=2\nu$ for all $n \ge 1$, and that $F_n$ converges in distribution towards a centered Gamma random variable $G(\nu)$. In addition, assume that either one of the following  conditions
\begin{itemize}
\item[(1)] \begin{equation}\label{eq:1/2-condition}
 \limsup_{n\to \infty} \frac{\sum_i c_{i,n} (c^3_{i,n} - c^2_{i,n})^2}{ \sum_i  (c^3_{i,n} - c^2_{i,n})^2} \le \frac{1}{2},
 \end{equation}
\item[(2)] the numerical sequence $\chi(n):= \sum_i c_{i,n} (c^3_{i,n} - c^2_{i,n})^2$ converges to $0$ from below;
\end{itemize}
are fulfilled. Then there exists a constant $C$, independent of $n$, such that 
\[
\Var \left( \Gamma_3 (F_n) - 2 \Gamma_2 (F_n) \right) \le_C \Var^{ \, 2} \left( \Gamma_1 (F_n) - 2 F_n \right).
\]
\item[\textit{(vi)}] In general, one has to note that the condition $$\text{Cov}\left( \Gamma_3 (F) - 2 \Gamma_2 (F), \Gamma_2 (F) - 2 \Gamma_1 (F)\right) = 0$$ does not guarantee that $F \stackrel{\mathcal{D}}{=} G(\nu)$. A simple counterexample is given by $F= c_1 (N^2_1 -1) + c_2 (N^2_2 -1)$ where (up to numerical error) $c_1= 1.27$ and $c_2= - 0.62$. In fact, $(c_1,c_2)$ is one of the intersections of two curves $x^2+y^2=2$, and $x^5 (x-1)^2 + y^5 (y-1)^2 =0$.
\end{itemize}

\end{rem}

The forthcoming results aim to provide neat characterizations of centered Gamma distribution $G(\nu)$, $\nu>0$ inside the second Wiener chaos by recruiting the theory of real quadratic forms.
\begin{lem}\label{lem:qf}
Let  $F= I_2 (f)$ be in the second Wiener chaos such that $\E[F^2]=2\nu$. For $\beta_1,\beta_2 \in \R$ define the following non-negative definite binary quadratic form 
\begin{equation}\label{eq:binary-qf}
\Psi_2 (\beta_1,\beta_2):= \Var \Big( \beta_1 \left(  \Gamma_3(F) - 2 \Gamma_2(F) \right) - 2 \beta_2 \left(  \Gamma_2(F) - 2 \Gamma_1(F) \right) \Big).
\end{equation}

 Moreover, assume that $\Delta = \text{discriminant } (\Psi_2) \neq 0$. Then $F \stackrel{\mathcal{D}}{=} G(\nu)$ if and only if the quadratic form $\Psi_2$ is isotropic, i.e.  $\Psi_2 (\beta_1,\beta_2)=0$ for some $(\beta_1,\beta_2) \neq (0,0)$. In fact, if $\Psi_2 (\beta_1,\beta_2)=0$ for some $(\beta_1,\beta_2) \neq (0,0)$, then $\Psi_2 = 0$ everywhere, and the random variable $F$ is distributed according to a centered Gamma distribution with parameter $\nu$.
\end{lem}

\begin{rem}
One has to note that, by the Cauchy-Schwarz inequality, $\Delta \ge 0$. Also, the  requirement $\Delta \neq 0$ is equivalent to saying that the form $\Psi_2$ is positive definite, i.e. $\operatorname{det} (A (\Psi_2)) >0$, where $A (\Psi_2)$ is the associated symmetric matrix. Also, when $\Delta = 0$, then $\Gamma_3 (F) - 2 \Gamma_2 (F) \stackrel{\text{a.s.}}{=} k \left( \Gamma_2(F) - 2 \Gamma_1 (F) \right)$, for some constant $k$, and therefore the binary quadratic form $\Psi_2$ as in \eqref{eq:binary-qf} reduces to $\Psi_2 (\beta_1,\beta_2) = (k\beta_1 - 2 \beta_2)^2 \Var \left( \Gamma_2 (F) - 2 \Gamma_1 (F) \right)$. In this case, the sole requirement $\Psi_2 (\beta_1,\beta_2) =0$ for some $(\beta_1,\beta_2) \neq (0,0)$ implies that, in general, $F$ is of the form given in \eqref{eq:mixed-gamma}.

\end{rem}
\begin{proof}[Proof of Lemma \ref{lem:qf}]:
Obviously, if $F \stackrel{\mathcal{D}}{=} G(\nu)$, then $\Psi_2 (\beta_1,\beta_2)=0$. Now assume that for some $(\beta_1,\beta_2) \neq (0,0)$, we have that $\Psi_2 (\beta_1,\beta_2)=0$. If $\beta_1 =0$, then, again, one can readily deduce that $F \stackrel{\mathcal{D}}{=} G(\nu)$. Otherwise 
\[
\Psi_2 (\beta_1, \beta_2)= \beta^2_1 \Var \Big( \left(  \Gamma_3(F) - 2 \Gamma_2(F) \right) - 2 \frac{\beta_2}{\beta_1}\left(  \Gamma_2(F) - 2 \Gamma_1(F) \right) \Big)=0,
\]
which immediately implies that $\Var(\Gamma_3(F) -2 \Gamma_2 (F))=0$, and hence $F \stackrel{\mathcal{D}}{=} G(\nu)$. 
\end{proof}

In general, for $s \ge 1$, put 
\begin{equation}\label{eq:all-forms1}
\mathscr{D}_s := \Big \{ \Psi_s (\beta_1, \cdots, \beta_s) := \Var \Big( \sum_{r \in A} \beta_r \left(  \Gamma_{r}(F) - 2 \Gamma_{r-1}(F)   \right) \Big) \, : A \subseteq \N  \text{ and } \# A =s \Big\}, 
\end{equation}
and
\[
\mathscr{D} = \cup_{s \ge 1} \mathscr{D}_s.
\]

\begin{lem}
Let  $F= I_2 (f)$ be in the second Wiener chaos such that $\E[F^2]=2\nu$. Then the set $\mathscr{D}$, containing non-negative quadratic forms of the form \eqref{eq:all-forms1}, characterizes the centered Gamma distribution in the sense that if for $ \Psi_s \in \mathscr{D}_s \subseteq \mathscr{D}$, and $s \ge 1$, 
\[
\Psi_s (\beta_1, \cdots, \beta_s)=0
\]
for some $(\beta_1,\cdots,\beta_s) \neq (0,\cdots,0)$, and if $\operatorname{det}(A(\Psi_s)) \neq 0$, then $F \sim CenteredGamma(\nu)$.
\end{lem}

\begin{proof}
Assume that $\Psi_s (\beta_1,\cdots,\beta_s)=0$ for some $(\beta_1,\cdots,\beta_s) \neq (0,\cdots,0)$. Note that $\Psi_s \ge 0$, and hence by \textit{Sylvester's law of inertia}, together with the condition $\operatorname{det}(A(\Psi_s)) \neq 0$, one can deduce that $\Psi_s = 0$. Therefore $\Var\left( \Gamma_r (F) - 2 \Gamma_{r-1}(F) \right) =0$ for any $r \in A$, where the set $A$ is same as \eqref{eq:all-forms1}, which implies that $F$ is distributed like $G(\nu)$.
\end{proof}

\begin{rem}\label{rem:general-form} 
Let $A \subset \N$ with $\# A = d < \infty$. The characterization of random elements $F$ in the second Wiener chaos such that for some $(\beta_1,\cdots,\beta_{d}) \in \R^{d}_{\neq 0}$, $$\sum_{r \in A} \beta_r \left( \bar{ \Gamma}_{r}(F) - 2 \bar{ \Gamma}_{r-1}(F)   \right) \stackrel{\text{a.s.}}{=} 0$$ is an interesting problem. The solution relates to the real roots of polynomial equations and the well-known \textit{ Abel--Ruffini theorem} (also known as \textit{Abel's impossibility theorem}) \cite{Abel-imposibility}, and we leave it for future investigation. For instance, when $\# A = 3$, the problem can be reduced to the real solutions of the \textit{trinomial} equation $x^n + a x + b =0$ for some $n \in \N$, and hence the \textit{Glasser's derivation} method can be useful \cite{glasser}. 

\end{rem}

\section{A New Proof for a Bound in Kolmogorov Distance}
 In this section, we use techniques that date back to Tikhomirov from 1981 \cite{Tikhomirov1981}, who used Stein's equation on the level of characteristic functions in order to present a result for Gamma approximation in terms of the \textit{Kolmogorov distance}. Similar lines of arguments have been recently employed in more generality in \cite{a-m-p-s}.\\

The starting point is the following classical Berry-Esseen lemma as stated in \cite[p.~104]{petrov-book}. For a more general version of the lemma, the reader is referred to Zolotarev \cite{zolot-65}.

\begin{Lem}\label{lem:B-E}
	Let $F$ and $G$ be two cumulative distribution functions with corresponding characteristic functions $\phi_F$, and $\phi_G$. Then for every positive number $T >0$, and every $b > 1/2\pi$, the estimate 
	\begin{equation}
	\begin{split}
	d_{\text{Kol}} (F,G):= \sup_{x \in \R} \abs{ F(x) - G(x) } \, \le \,  & b \int_{-T}^{T} \abs[\Big]{ \frac{\phi_F (t) - \phi_G(t)}{t} } \,dt \\
	& \quad + b \, T \, \sup_x \int_{ \abs{y} \le \frac{c(b)}{T}} \abs{ G(x+y) - G(x) } \, dy,
	\end{split} \label{eq:EsseenLemma}
	\end{equation}
	takes place, where $c(b)$ is a constant depending only on $b$, and it is given by the root of the following equation
	\[ \int_{0}^{\frac{c(b)}{4}} \frac{\sin ^2 u}{u^2} \, du = \frac{\pi}{4} + \frac{1}{8b}. \]
	In particular, if $\sup_x \abs{ G'(x) } \le K$, then 
	\begin{equation}\label{eq:B-E-bounded}
	\begin{split}
	d_{\text{Kol}} (F,G):= \sup_{x \in \R} \abs{ F(x) - G(x)  } \, \le \,   b \int_{-T}^{T} \abs[\Big]{ \frac{\phi_F (t) - \phi_G(t)}{t} } \, dt
	+ c(b) \frac{K}{T}.
	\end{split}
	\end{equation}
\end{Lem}

In order to prove a Kolmogorov bound, we need an estimate on the difference of the characteristic functions and the distribution functions. The first is done in the following Lemma:

\begin{Lem}\label{lem:CharacteristicFuctionDifferenceBound}
	Let $\nu >0$ be an integer and let $F$ be a random variable admitting a finite chaos expansion with variance $\E[F^2] = 2 \nu$. Let $G(\nu) \sim CenteredGamma(\nu)$. Define
	\[ D(t) := \phi_F(t) - \phi_{G(\nu)}(t)  = \E\big[ e^{i t F} \big] - \E \big[ e^{i t G(\nu)} \big], \qquad t \in \R.  \]
	Then the following estimates take place:
	\begin{equation}
	\abs{D(t)} \leq \frac{1}{2} \, \abs{t} \, \E \abs{ 2(F+\nu) - \Gamma_1(F) } \leq \frac{1}{2} \, \abs{t} \, \sqrt{\Var(\Gamma_1(F) - 2F)}. \label{eq:BoundOnCharFunc-Gamma}
	\end{equation}
\end{Lem}

\begin{proof}
	We consider the Stein operator associated to a centered Gamma random variable $G(\nu)$ (see \cite{d-p}, equation 2.7):
	\[
	\mathfrak{L}f(x) = 2 (x+ \nu) f'(x) - x f(x).
	\]
	Using the integration by parts formula, we get for all $f \in C^1$ with bounded derivative
	\begin{equation}
	\E[ \mathfrak{L}f(F) ] = \E\left[ f'(F) \big\{  2(F+ \nu) - \Gamma_1(F) \big\} \right]. \label{eq:GammaSteinIntByParts}
	\end{equation}
	Also, for all $C^1$ functions $f:\R \to \R$, such that the expectation exists (e.g. if $f$ is polynomially bounded), we have
	\begin{equation}
	\E[ \mathfrak{L}f(G(\nu)) ] = 0. \label{eq:GammaSteinExpectation}
	\end{equation}
	By considering real and imaginary part separately and using linearity, we can extend \eqref{eq:GammaSteinIntByParts} and \eqref{eq:GammaSteinExpectation} to complex valued functions $f: \R \to \C$. Thus letting $f(x) = e^{i t x}$ for $t \in \R$, we obtain
	\[
	\E[ \mathfrak{L}f(F) ] = it \, \E \big[ e^{itF} \big\{ 2(F+\nu) - \Gamma_1(F)  \big\} \big].
	\]
	Therefore
	\begin{align*}
	& it \, \E \big[ e^{itF} \big\{ 2(F+\nu) - \Gamma_1(F)  \big\} \big] = \E[ \mathfrak{L}f(F) ] = \E[ \mathfrak{L}f(F) ] - 0 = \E[ \mathfrak{L}f(F) ] - \E[ \mathfrak{L}f(G(\nu)) ] \\
	& \quad = it\times 2 \nu \Big( \E\big[e^{itF}\big] - \E\big[e^{itG(\nu)} \big] \Big) - (1-2it) \Big(  \E\big[Fe^{itF}\big] - \E\big[G(\nu)e^{itG(\nu)} \big] \Big) \\
	& \quad = it \times 2\nu D(t) + (2t+i) D'(t).
	\end{align*}
	So $D$ satisfies the differential equation
	\begin{equation}
	(1-2ti) D'(t) + 2 \nu \, t D(t) = e(t), \quad \text{where } e(t) :=t \, \E \big[ e^{itF} \big\{ 2(F+\nu) - \Gamma_1(F)  \big\} \big]. \label{eq:GammaODE} 
	\end{equation}
	Using the fact that $D(-t)=\overline{D(t)}$ and $\abs{D(t)} = \abs{\overline{D(t)}}$, we focus only on $t \geq 0$. The solution of the ordinary differential equation \eqref{eq:GammaODE} with initial condition $D(0)=0$ is given by
	\[
	D(t) = e^{-a(t)} \int_{0}^{t} \frac{e(s)}{1-2si} \, e^{a(s)} \, ds,
	\]
	where
	\[
	a(t) = \int \frac{2 \nu t}{1-2ti} \, dt = \frac{\nu}{4} \, \log(4t^2+1) + i \, \Big(t \nu - \frac{\nu}{2} \, \operatorname{arctan}(2t) \Big).
	\]
	Note that
	\[
	\abs{e^{a(t)}} = (4t^2+1)^{\frac{\nu}{4}} \quad \text{and} \quad \abs{e^{-a(t)}} = (4t^2+1)^{- \frac{\nu}{4}}.
	\]
	Thus we can estimate
	\begin{align*}
	\abs{D(t)} & \leq \abs{e^{-a(t)}} \int_{0}^{t} \abs*{\frac{1}{1-2si}} \abs{e(s)} \, \abs{e^{a(s)}} \, ds \\
	& \leq  (4t^2+1)^{- \frac{\nu}{4}} \int_{0}^{t} \frac{(4s^2+1)^{\nu/4}}{\sqrt{4s^2+1}} \, \abs{e(s)} \, ds \\
	& \leq \E \big[ \abs{ 2(F+\nu) - \Gamma_1(F) } \big] \, (4t^2+1)^{- \frac{\nu}{4}} \int_{0}^{t} s\,  (4s^2+1)^{\frac{\nu}{4} - \frac{1}{2}} \, ds \\
	& = \E \big[ \abs{ 2(F+\nu) - \Gamma_1(F) } \big] \, (4t^2+1)^{- \frac{\nu}{4}} \bigg( \frac{1}{2(\nu+2)} \Big[ (4t^2+1)^{\frac{\nu}{4} + \frac{1}{2}} -1 \Big] \bigg) \\
	& = \E \big[ \abs{ 2(F+\nu) - \Gamma_1(F) } \big] \, \frac{1}{2(\nu+2)} \, \big( \sqrt{4t^2+1} - (4t^2+1)^{-\nu/4} \big) \\
	& \leq \frac{1}{2} \, t \, \E \big[ \abs{ 2(F+\nu) - \Gamma_1(F) } \big].
	\end{align*}
	The last estimate is due to Lemma \ref{Lem:NuT-Estimate} below. The second inequality in \eqref{eq:BoundOnCharFunc-Gamma} is just Cauchy Schwarz.
\end{proof}

\begin{Lem} \label{Lem:NuT-Estimate}
	For any $\nu > 0$ and $t\geq 0$, we have that
	\[ \sqrt{4t^2 +1} - (4t^2+1)^{- \nu/4} \leq (2+ \nu) \times t. \]
\end{Lem}
\begin{proof}
	We make use of the following well-known inequalities:
	\begin{align}
	\sqrt{x+y} \leq \sqrt{x} + \sqrt{y}, \quad & \text{for all } x,y\geq 0; \label{eq:SqrtIneq} \\
	1-e^{-x} \leq x, \quad & \text{for all } x \geq -1; \label{eq:ExpIneq} \\
	\log(x) \leq 2(\sqrt{x} -1), \quad & \text{for all } x > 0. \label{eq:LogIneq}
	\end{align}
	With this we get
	\begin{align*}
	\sqrt{4t^2 +1} - (4t^2+1)^{- \nu/4} & \stackrel{\eqref{eq:SqrtIneq}}{\leq} 2t + 1 - e^{- \frac{\nu}{4} \log(4t^2+1)} \\
	& \stackrel{\eqref{eq:ExpIneq}}{\leq} 2t + \frac{\nu}{4} \, \log(4t^2+1) \\
	& \stackrel{\eqref{eq:LogIneq}}{\leq} 2t + \frac{\nu}{2} \, (\sqrt{4t^2+1} -1) \\
	& \stackrel{\eqref{eq:SqrtIneq}}{\leq} (2+\nu) \times t.
	\end{align*}
\end{proof}

In order to estimate the second term in the Esseen-Lemma \eqref{eq:EsseenLemma}, we need to study the cumulative distribution function (CDF) of a centered Gamma random variable $G(\nu)$. We show that it is Hölder-continuous with a Hölder-exponent depending on $\nu$.

\begin{Lem}\label{lem:CDF-Behaviour}
	Let $\nu>0$ be an integer and $G(\nu) \sim CenteredGamma(\nu)$. Denote by $G_{\nu}$ its CDF and by $g_{\nu}$ its probability density function (PDF). Then there exists a constant $K>0$, such that for all $a,b \in \R$ we have
	\begin{equation}
	\abs{G_{\nu}(a) - G_{\nu}(b)} \leq K \, |a-b|, \quad \text{if } \nu \geq 2, \label{eq:CDF-Behaviour1}
	\end{equation}
	and
	\begin{equation}
	\abs{G_{\nu}(a) - G_{\nu}(b)} \leq K \, |a-b|^{1/2}, \quad \text{if } \nu = 1. \label{eq:CDF-Behaviour2}
	\end{equation}
\end{Lem}

\begin{proof}
	The PDF of $G(\nu)$ is given by
	\[ g_\nu(x) = 2^{-\frac{\nu}{2}}\, \Gamma\Big(\frac{\nu}{2} \Big)^{-1} (x+\nu)^{\frac{\nu}{2} -1} \, e^{- \frac{x}{2} - \frac{\nu}{2}} \, \ind{x > -\nu}(x). \]
	If $\nu \geq 3$, then $g_\nu$ is continuous and hence $G_\nu$ is differentiable on the whole real line. One can readily verify that $g_\nu$ is bounded with $K:= \sup_{x \in \R} \abs{g_\nu(x)} = g_\nu(-2)$. So \eqref{eq:CDF-Behaviour1} is just an application of the mean value theorem.\\
	\\
	When $\nu=2$, then
	\[ g_\nu(x) = \frac{1}{2} \, e^{-\frac{x}{2} -1} \,  \ind{x>-2}(x).  \]
	In this case $G_\nu$ is not differentiable in $x=-2$. However, because of the monotonicity, $g_\nu$ is bounded by $K:= \sup_{x \in \R} \abs{g_\nu(x)} = \lim_{x \downarrow -2} g_\nu(x)= 1/2$. Therefore \eqref{eq:CDF-Behaviour1} holds for all $a,b \in (-\infty, -2)$ and all $a,b \in (-2, \infty)$. Using the continuity of $G_\nu$, we can easily show that \eqref{eq:CDF-Behaviour1} extends to the whole real line.
	\\
	When $\nu=1$, the PDF has the form
	\[ g_\nu(x) = \frac{1}{\sqrt{2 \pi}} \, \frac{e^{-\frac{x}{2}- \frac{1}{2}}}{\sqrt{x+1}} \, \ind{x>-1}(x). \]
	First note that $g_\nu$ is not bounded, in fact $\lim_{x \downarrow -1} g_\nu(x) = \infty$. Without loss of generality, assume that $b>a$. We split the proof into three cases:\\
	\\
	\underline{Case 1 ($a < b \leq -1$)}: Here \eqref{eq:CDF-Behaviour2} holds, since $G_\nu(a)=G_\nu(b)=0$.\\
	\\
	\underline{Case 2 ($-1 < a < b $)}: Define $C := \frac{1}{\sqrt{2 \pi}}$. Then we have
	\[ g_\nu(x) \leq C \times \frac{1}{\sqrt{x+1}} \quad \forall x > -1. \]
	We compute (note that $G_\nu$ is increasing):
	\[
	G_\nu(b) - G_\nu(a) = \int_{a}^{b} g_\nu(t) \, dt \leq C \, \int_{a}^{b} \frac{dt}{\sqrt{t+1}} = 2 \, C \, (\sqrt{b+1} - \sqrt{a+1}) \leq 2 \, C \, \sqrt{b-a}.
	\]
	\underline{Case 3 ($a \leq -1 < b $)}: Using the continuity of $G_\nu$ we get:
	\begin{align*}
	G_\nu(b) - G_\nu(a) & = G_\nu(b) - G_\nu(-1)  = \lim_{\epsilon \downarrow -1}  G_\nu(b) - G_\nu(-1 + \epsilon) \\
	& \hspace{-.7em} \stackrel{\text{Case 2}}{\leq} \lim_{\epsilon \downarrow -1} 2 \, C \, \sqrt{b+1- \epsilon} = 2 \, C \, \sqrt{b+1} \leq 2 \, C \, \sqrt{b-a}.\\
	\end{align*}
\end{proof}

\begin{Rem}
	Now let $a=-1$ and $b \in (-1,0)$. With similar arguments as above, this time using the upper bound $g_\nu(t) \geq \frac{e^{-1/2}}{\sqrt{2 \pi}} \times \frac{1}{\sqrt{t+1}}$, we can show that
	\[ G_\nu(b) - G_\nu(-1) \geq \frac{\sqrt{2} \, e^{-1/2}}{\sqrt{\pi}} \times \sqrt{b - (-1)}.  \]
	Thus in a vicinity of $-1$, estimate \eqref{eq:CDF-Behaviour2} is actually the best we can do when $\nu=1$.
\end{Rem}

Now we have all the ingredients to show the following theorem.
\begin{Thm}
	Let $\nu>0$ be an integer and let $F$ be a random variable admitting a finite chaos expansion, such that $\E[F^2]=2\nu$. Let $G(\nu) \sim CenteredGamma(\nu)$. Then
	\[
	 d_{Kol}(F,G(\nu)) \leq \begin{cases}
		  C \times \Var \big(\Gamma_1(F) -2 F \big)^{\frac{1}{4}}, & \text{if } \nu \geq 2 \\
		  C \times \Var \big(\Gamma_1(F) -2 F \big)^{\frac{1}{6}}, & \text{if } \nu =1,
	 \end{cases}
	 \]
	 where $C>0$ is a constant only depending on $\nu$.
\end{Thm}

\begin{proof}
	If $\nu \geq 2$ then putting the bounds from Lemma \ref{lem:CharacteristicFuctionDifferenceBound} and Lemma \ref{lem:CDF-Behaviour} into the Berry-Esseen lemma \eqref{eq:EsseenLemma}, we get for every $T>0$
	\begin{align*}
	d_{Kol}(F,G(\nu)) & \leq  b \int_{-T}^{T} \abs*{ \frac{\phi_F (t) - \phi_{G(\nu)}(t)}{t} } \,dt + b \, T \, \sup_{x \in \R} \int_{ \abs{ y } \le \frac{c(b)}{T}} \abs*{ G_{\nu}(x+y) - G_{\nu}(x) } \, dy \\
	& \leq b \, T \, \sqrt{ \Var \big( \Gamma_1(F) - 2F \big) } + b \,  K \, \frac{c(b)^2}{T} \\
	& =: c_1 \, T \sqrt{ \Var \big( \Gamma_1(F) - 2F \big) } + \frac{c_2}{T}.
	\end{align*}
	The minimum is achieved at
	\[ T_{\min}= \left( \frac{c_2}{c_1} \right)^{1/2} \Var\big( \Gamma_1(F) - 2F \big)^{-1/4}, \]
	and is given by
	\[ 2\, \sqrt{c_1 c_2} \, \Var\big( \Gamma_1(F) - 2F \big)^{1/4}.  \]
	If $\nu = 1$, then instead we get
	\begin{align*}
	d_{Kol}(F,G(\nu)) & \leq  b \int_{-T}^{T} \abs*{ \frac{\phi_F (t) - \phi_{G(\nu)}(t)}{t} } \,dt + b \, T \, \sup_{x \in \R} \int_{ \abs{ y } \le \frac{c(b)}{T}} \abs*{ G_{\nu}(x+y) - G_{\nu}(x) } \, dy \\
	& \leq b \, T \, \sqrt{ \Var \big( \Gamma_1(F) - 2F \big) } + \frac{4}{3} b \,  K \, \frac{c(b)^{3/2}}{T^{1/2}} \\
	& =: \tilde{c}_1 \, T \sqrt{ \Var \big( \Gamma_1(F) - 2F \big) } + \frac{\tilde{c}_2}{T^{1/2}}.
	\end{align*}
	Again, minimizing over $T>0$ yields
	\[ T_{\min}= 2^{-2/3} \, \left(\frac{\tilde{c}_2}{\tilde{c}_1}\right)^{2/3} \, \Var(\Gamma_1(F) - 2F)^{-1/3}  \]
	and thus the minimum is
	\[ 3 \times 2^{-2/3} \times \tilde{c}_1^{1/3} \tilde{c}_2^{2/3}  \Var( \Gamma_1(F) - 2F)^{1/6}.  \]
\end{proof}

\begin{Rem}
	Most parts of this result are not new, we merely present an original proof to illustrate the power of other techniques that are mostly not relying on Stein's method. In fact, using Theorem 1.7 from \cite{d-p}, as well as the fact that
	\[ d_{Kol}(F,G) \leq C \, \sqrt{ d_{1}(F,G) }, \]
	whenever the density of $G$ is bounded, we immediately retrieve the case $\nu \geq 2$. To our best knowledge, when $\nu<2$, our result is new, as in this case the corresponding density $g_1$ is not  bounded.
\end{Rem}
Since we were mainly interested in $F$ belonging to the second Wiener chaos, we have only focussed on integer valued $\nu$. However, the proofs can easily be adapted to cover any $\nu>0$, which leads to the following generalization.
\begin{Thm}
	Let $\nu>0$ be any positive real number and let $F$ be a random variable admitting a finite chaos expansion, such that $\E[F^2]=2\nu$. Let $G(\nu) \sim CenteredGamma(\nu)$. Then
	\[
	d_{Kol}(F,G(\nu)) \leq \begin{cases}
	C \times \Var \big(\Gamma_1(F) -2 F \big)^{\frac{1}{4}}, & \text{if } \nu \geq 2 \\
	C \times \Var \big(\Gamma_1(F) -2 F \big)^{\frac{\nu}{2(\nu + 2)}}, & \text{if } \nu  \in (0,2),
	\end{cases}
	\]
	where $C>0$ is a constant only depending on $\nu$.

\end{Thm}

Under the light of the result presented in Section \ref{sec:concrete-example}, we end the paper with the following conjecture.
\begin{con}\label{con:<>}
Let $\nu>0$, and $F=I_2 (f)$ belonging to the second Wiener chaos so that $\E[F^2]=2\nu$. Let $G(\nu) \sim CenteredGamma(\nu)$. Then there exist two general constants $0 < C_1< C_2$ such that 
\begin{equation}\label{eq:<>}
C_1 \mathbf{M}(F) \le d_{TV} (F,G(\nu)) \le C_2 \mathbf{M}(F).
\end{equation} 
\end{con}

\section{Appendix}

The following lemma provides an explicit representation of the new Gamma operators used in this paper in terms of contractions. Recall that these are not the same as e.g. in \cite{CumOnTheWienerSpace}, but rather the new ones introduced in \eqref{eq:GammOperatorDefinition}.

\begin{lem} \label{lem:RepresentationOfGamma_j}
	For $q \geq 1$, let $F=I_q(f)$, for some $f \in \HH^{\odot q}$ be an element of the $q$-th Wiener chaos. Then
	\begin{align}
	\Gamma_{s}(F)  = & \sum_{r_1=1}^{q} \cdots \sum_{r_{s}=1}^{[sq - 2 r_1 - \cdots - 2 r_{s-1}] \wedge q} c_q (r_1, \ldots, r_{s}) \ind{r_1<q} \ldots \ind{r_1 + \cdots + r_{s-1} < \frac{sq}{2}} \notag \\
	& \times I_{(s+1)q - 2 r_1 - \cdots - 2 r_{s} } \left( \left( \left( \ldots (f \scont{r_1} f ) \scont{r_2} f  \right) \ldots f  \right) \scont{r_{s}} f \right). \label{eq:NewGammaOpProof1}
	\end{align}
\end{lem}
\begin{proof}
	Without loss of generality, we assume that $\mathfrak{H} = L^2(\mathbb{A}, \mathscr{A}, \mu)$, where $(\mathbb{A},\mathscr{A})$ is a measurable space and $\mu$ a $\sigma$-finite measure without atoms. \\
	Note that for $s=1$, the product $\ind{r_1<q} \ldots \ind{r_1 + \cdots + r_{s-1} < \frac{sq}{2}}$ is empty, i.e. $1$. In this case \eqref{eq:NewGammaOpProof1} reads:
	\[ \Gamma_1(F) =  \sum_{r=1}^{q} c_q(r) I_{2q-2r} ( f \scont{r} f).  \]
	We show this using the product formula:
	\begin{align*}
	\Gamma_1(F)  & = \sprod{DF, -DL^{-1} F}_\mathfrak{H} = \frac{1}{q} \norm{DF}_\mathfrak{H}^2 =  q \int_{\mathbb{A}} I_{q-1} \left( f(\cdot,a) \right)^2 \, \mu (da) \\
	& = q \sum_{r=0}^{q-1} r! \binom{q-1}{r}^2 I_{2q-2r-2} \left( \int_{\mathbb{A}} f(\cdot, a) \scont{r} f(\cdot,a) \, \mu (da) \right) \\
	& = q \sum_{r=0}^{q-1} r! \binom{q-1}{r}^2 I_{2q-2r-2} ( f \scont{r+1} f ) \\
	& = q \sum_{r=1}^{q} (r-1)! \binom{q-1}{r-1}^2 I_{2q-2r} (f \scont{r} f).
	\end{align*}
	We now show the induction step $s-1 \to s$. Note that we have
	\[   -DL^{-1}(F)(a)  = -DL^{-1}(I_q(f))(a) = -D \left( - \frac{1}{q} I_q(f) \right)(a) = I_{q-1}(f(\cdot , a)) \]
	and
	\[ D I_p((f)) (a) = \ind{p>0} \, p \, I_{p-1}(f(\cdot, a)). \]
	Therefore
	\begin{align*}
	D \Gamma_{s-1}(F) (a) = & \sum_{r_1=1}^{q} \cdots \sum_{r_{s-1}=1}^{[(s-1)q - 2 r_1 - \cdots - 2 r_{s-2}] \wedge q} c_q (r_1, \ldots, r_{s-1}) \ind{r_1<q} \ldots \ind{r_1 + \cdots + r_{s-2} < \frac{(s-1)q}{2}} \notag \\
	& \times \ind{r_1 + \cdots + r_{s-1} < \frac{sq}{2}} \, (sq - 2 r_1 - \cdots - 2 r_{s-1}) \\
	& \times I_{sq - 2 r_1 - \cdots - 2 r_{s-1} -1 } \left( \left( \left[ \left[ \ldots [f \scont{r_1} f ] \scont{r_2} f  \right] \ldots f  \right] \scont{r_{s-1}} f \right) (\cdot, a) \right),
	\end{align*}
	and thus
	\begin{align*}
	& \Gamma_s(F) =  \sprod{D \Gamma_{s-1}(F), -DL^{-1} F}_\mathfrak{H} \\ 
	{}={} &  \sum_{r_1=1}^{q} \cdots \sum_{r_{s-1}=1}^{[(s-1)q - 2 r_1 - \cdots - 2 r_{s-2}] \wedge q}  \, c_q (r_1, \ldots, r_{s-1}) \ind{r_1<q} \ldots \ind{r_1 + \cdots + r_{s-2} < \frac{(s-1)q}{2}} \notag \\
	& \times \ind{r_1 + \cdots + r_{s-1} < \frac{sq}{2}} \, (sq - 2 r_1 - \cdots - 2 r_{s-1})  \\
	& \times \int_{\mathbb{A}} I_{q-1}(f(\cdot, a)) \, I_{sq - 2 r_1 - \cdots - 2 r_{s-1} -1 } \left( \left( \left[ \left[ \ldots [f \scont{r_1} f ] \scont{r_2} f  \right] \ldots f  \right] \scont{r_{s-1}} f \right) (\cdot, a) \right) \, \mu (da) \\
	{}={} &  \sum_{r_1=1}^{q} \cdots \sum_{r_{s-1}=1}^{[(s-1)q - 2 r_1 - \cdots - 2 r_{s-2}] \wedge q} \, c_q (r_1, \ldots, r_{s-1}) \ind{r_1<q} \ldots \ind{r_1 + \cdots + r_{s-2} < \frac{(s-1)q}{2}} \notag \\
	& \times \ind{r_1 + \cdots + r_{s-1} < \frac{sq}{2}} \, (sq - 2 r_1 - \cdots - 2 r_{s-1})  \\
	& \times \sum_{r_s=1}^{[sq-2r_1- \cdots - 2r_{s-1} ] \wedge q} (r_s-1)! \, \binom{sq-2r_1- \cdots -  2r_{s-1} -1 }{r_s - 1} \binom{q-1}{r_s - 1} \\
	& \hspace{-.9em} I_{(s+1)q-2r_1 - \cdots - 2r_{s-1} - 2 r_{s}} \left( \int_{\mathbb{A}} \left( \left[ \left[ \ldots [f \scont{r_1} f ] \scont{r_2} f  \right] \ldots f  \right] \scont{r_{s-1}} f \right) (\cdot, a)  \scont{r_s -1} f(\cdot,a) \, \mu (da) \right) \\
	{}={} &  \sum_{r_1=1}^{q} \cdots \sum_{r_{s-1}=1}^{[(s-1)q - 2 r_1 - \cdots - 2 r_{s-2}] \wedge q} \, c_q (r_1, \ldots, r_{s-1}) \ind{r_1<q} \ldots \ind{r_1 + \cdots + r_{s-2} < \frac{(s-1)q}{2}} \notag \\
	& \times \ind{r_1 + \cdots + r_{s-1} < \frac{sq}{2}} \, (sq - 2 r_1 - \cdots - 2 r_{s-1}) \\
	& \times \sum_{r_s=1}^{[sq-2r_1- \cdots - 2r_{s-1} ] \wedge q} (r_s-1)! \, \binom{sq-2r_1- \cdots -  2r_{s-1} -1 }{r_s - 1} \binom{q-1}{r_s - 1} \\
	& I_{(s+1)q-2r_1 - \cdots - 2r_{s-1} - 2 r_{s}} \left( \left( \left[ \left[ \ldots [f \scont{r_1} f ] \scont{r_2} f  \right] \ldots f  \right] \scont{r_{s-1}} f \right) \scont{r_s} f \right). \\
	\end{align*}
	The constants are recursively defined via
	\begin{flalign*}
	& c_q(r) = q \, (r-1)! \, \binom{q-1}{r-1}^2, &
	\end{flalign*}
	and
	\begin{align} 
	& c_q(r_1, \cdots, r_{s}) =  \notag \\
	& (sq- 2 r_1 - \cdots - 2 r_{s-1}) \, (r_s-1)!  \, \binom{sq-2r_1- \cdots -  2r_{s-1} -1 }{r_s - 1} \binom{q-1}{r_s - 1} \,  c_q(r_1,  \cdots, r_{s-1}). \label{eq:RecursiveFormulaForc_q}
	\end{align}
\end{proof}

With this, we are able to proof Proposition \ref{Prop:RelationOldAndNewGamma}

\begin{proof}[Proof of Proposition \ref{Prop:RelationOldAndNewGamma}]
	Part \textit{(a)} is clear from the definition. Part \textit{(b)} for $j = 1$ is also trivial. For $j = 2$, we use the fact that $\Gamma_1 = \Gamma_{alt,1}$, as well as the integration by parts formula \eqref{eq:IntegrationByParts}, to get
	\begin{align*}
	\E \big[ \Gamma_2(F) \big] & = \E \big[ \sprod{D \Gamma_1(F), -D L^{-1} F}_{\HH} \big] = \E\big[ \Gamma_1(F) \, F \big] \\
	& = \E \big[F \, \Gamma_{alt,1}(F) \big] =  \E \big[ \sprod{D F, -D L^{-1} \Gamma_{alt,1}(F)}_{\HH} \big] = \E[\Gamma_{alt,2}(F)].
	\end{align*}
	For part \textit{(c)}, consider
	\begin{align*}
	\E \big[ \Gamma_3(F) \big] & = \E \big[ \sprod{D \Gamma_2(F), -D L^{-1} F}_{\HH} \big] = \E \big[ F \, \Gamma_2(F) \big] = \E \big[ F \, \sprod{D \Gamma_1(F), -D L^{-1} F}_{\HH} \big] \\
	& = \E \big[ \sprod{D \big(F \,\Gamma_1(F) \big), -D L^{-1} F}_{\HH} \big] - \E \big[ \Gamma_1(F) \sprod{D F, -D L^{-1} F}_{\HH} \big] \\
	& = \E \big[ \sprod{D \big(F \,\Gamma_1(F) \big), -D L^{-1} F}_{\HH} \big] - \E \big[ \Gamma_{alt,1}(F)^2 \big] \\
	& = \E \big[ F^2 \, \Gamma_{alt,1} \big] - \E \big[ \Gamma_{alt,1}(F)^2 \big] \\
	& = \E[F^2] \, \E\big[ \Gamma_{alt,1}(F) \big] + \E \big[ 2F \, \sprod{D F, -D L^{-1}  \Gamma_{alt,1}(F)}_{\HH}  \big] - \E \big[ \Gamma_{alt,1}(F)^2 \big] \\
	& = \E\big[ \Gamma_{alt,1}(F) \big]^2 + 2 \, \E \big[ F \, \Gamma_{2,alt} \big] - \E \big[ \Gamma_{alt,1}(F)^2 \big] \\
	& = - \Var\big(\Gamma_{alt,1}(F)\big) + 2 \, \E\big[ \Gamma_{alt,3}(F) \big].
	\end{align*}
	For part \textit{(d)}, we consider the representation of $\Gamma_{alt,s}$ given in equation (5.25) of \cite{CumOnTheWienerSpace}. The representation is exactly the same as for $\Gamma_s$ (Lemma \ref{lem:RepresentationOfGamma_j}), except for the recursive formula of the constants $c_q$. For $\Gamma_{alt,j}$ they are given by
	\[ c_{alt,q}(r) = c_q(r) = q (r-1)! \binom{q-1}{r-1}^2  \]
	and
	\[ c_{alt,q}(r_1, \ldots, r_s) = q \, (r_s-1)!  \, \binom{sq-2r_1- \cdots -  2r_{s-1} -1 }{r_s - 1} \binom{q-1}{r_s - 1} \,  c_q(r_1,  \cdots, r_{s-1}). \]
	Comparing this with our formula \eqref{eq:RecursiveFormulaForc_q}, we see that only the first factor is different, namely $q$ instead of $(sq-2r_1-\ldots - 2r_{s-1})$. But now for $q=2$, the indicator $\ind{r_1 + \cdots + r_{s-1} < \frac{sq}{2}}$ dictates that $r_1 = \ldots = r_{s-1} = 1$ and hence
	\[ q = 2 = 2s-2r_1 - \ldots - 2 r_{s-1}. \]
	Therefore, the two notions of Gamma operators coincide when $q=2$.
	\end{proof}

\section*{Acknowledgments}
The authors would like to thank Simon Campese for pointing out a mistake in the proof of Theorem \ref{thm:MainMalliavinSteinBound} in an earlier version of this preprint.


\end{document}